\documentclass{amsart}
\usepackage{amsfonts}
\usepackage{amssymb, amsmath, amsthm}
\usepackage{amscd}
\usepackage{verbatim}
\usepackage{enumerate}

\allowdisplaybreaks
\newtheorem{theorem}{Theorem}[section]

\newtheorem{definition}[theorem]{Definition}

\newtheorem{lemma}[theorem]{Lemma}

\newtheorem{proposition}[theorem]{Proposition}

\theoremstyle{remark}
\newtheorem{remark}[theorem]{Remark}
\newtheorem{example}[theorem]{Example}
\numberwithin{equation}{section}

\newcommand{\eps}{\varepsilon}
\def\typeout#1{\message{^^J}\message{#1}\message{^^J}}
\newif\ifSRCOK \SRCOKtrue
\newcount\PAGETOP
\newcount\LASTLINE
\global\PAGETOP=1
\global\LASTLINE=-1
\def\EJECT{\SRC\eject}
\def\WinEdt#1{\typeout{:#1}}
\gdef\MainFile{\jobname.tex}
\gdef\CurrentInput{\MainFile}
\newcount\INPSP
\global\INPSP=0
\def\SRC{\ifSRCOK  \ifnum\inputlineno>\LASTLINE    \ifnum\LASTLINE<0      \global\PAGETOP=\inputlineno    \fi    \global\LASTLINE=\inputlineno    \ifnum\INPSP=0      \ifnum\inputlineno>\PAGETOP        
      \fi    \else      
    \fi  \fi\fi}
\def\PUSH#1{\SRC\ifnum\INPSP=0 \global\let\INPSTACKA=\CurrentInput \else\ifnum\INPSP=1 \global\let\INPSTACKB=\CurrentInput \else\ifnum\INPSP=2 \global\let\INPSTACKC=\CurrentInput \else\ifnum\INPSP=3 \global\let\INPSTACKD=\CurrentInput \else\ifnum\INPSP=4 \global\let\INPSTACKE=\CurrentInput \else\ifnum\INPSP=5 \global\let\INPSTACKF=\CurrentInput \else               \global\let\INPSTACKX=\CurrentInput \fi\fi\fi\fi\fi\fi\gdef\CurrentInput{#1}\WinEdt{<+ \CurrentInput}\global\LASTLINE=0\ifSRCOK\fi\global\advance\INPSP by 1}
\def\POP{\ifnum\INPSP>0 \global\advance\INPSP by -1  \fi\ifnum\INPSP=0 \global\let\CurrentInput=\INPSTACKA \else\ifnum\INPSP=1 \global\let\CurrentInput=\INPSTACKB \else\ifnum\INPSP=2 \global\let\CurrentInput=\INPSTACKC \else\ifnum\INPSP=3 \global\let\CurrentInput=\INPSTACKD \else\ifnum\INPSP=4 \global\let\CurrentInput=\INPSTACKE \else\ifnum\INPSP=5 \global\let\CurrentInput=\INPSTACKF \else               \global\let\CurrentInput=\INPSTACKX \fi\fi\fi\fi\fi\fi\WinEdt{<-}\global\LASTLINE=\inputlineno\global\advance\LASTLINE by -1\SRC}
\def\INPUT#1{\relax}

\let\originalxxxeverypar\everypar
\newtoks\everypar
\originalxxxeverypar{\the\everypar\expandafter\SRC}
\everymath\expandafter{\the\everymath\expandafter\SRC}
\output\expandafter{\expandafter\SRCOKfalse\the\output}
\newif\ifSRCOK \SRCOKtrue
\newcount\PAGETOP
\newcount\LASTLINE
\global\PAGETOP=1
\global\LASTLINE=-1
\gdef\MainFile{\jobname.tex}
\gdef\CurrentInput{\MainFile}
\newcount\INPSP
\global\INPSP=0
\def\EJECT{\SRC\eject}
\def\WinEdt#1{\typeout{:#1}}
\def\SRC{\ifSRCOK  \ifnum\inputlineno>\LASTLINE    \ifnum\LASTLINE<0      \global\PAGETOP=\inputlineno    \fi    \global\LASTLINE=\inputlineno    \ifnum\INPSP=0      \ifnum\inputlineno>\PAGETOP              \fi    \else          \fi  \fi\fi}
\def\PUSH#1{\SRC\ifnum\INPSP=0 \global\let\INPSTACKA=\CurrentInput \else\ifnum\INPSP=1 \global\let\INPSTACKB=\CurrentInput \else\ifnum\INPSP=2 \global\let\INPSTACKC=\CurrentInput \else\ifnum\INPSP=3 \global\let\INPSTACKD=\CurrentInput \else\ifnum\INPSP=4 \global\let\INPSTACKE=\CurrentInput \else\ifnum\INPSP=5 \global\let\INPSTACKF=\CurrentInput \else               \global\let\INPSTACKX=\CurrentInput \fi\fi\fi\fi\fi\fi\gdef\CurrentInput{#1}\WinEdt{<+ \CurrentInput}\global\LASTLINE=0\ifSRCOK\fi\global\advance\INPSP by 1}
\def\POP{\ifnum\INPSP>0 \global\advance\INPSP by -1  \fi\ifnum\INPSP=0 \global\let\CurrentInput=\INPSTACKA \else\ifnum\INPSP=1 \global\let\CurrentInput=\INPSTACKB \else\ifnum\INPSP=2 \global\let\CurrentInput=\INPSTACKC \else\ifnum\INPSP=3 \global\let\CurrentInput=\INPSTACKD \else\ifnum\INPSP=4 \global\let\CurrentInput=\INPSTACKE \else\ifnum\INPSP=5 \global\let\CurrentInput=\INPSTACKF \else               \global\let\CurrentInput=\INPSTACKX \fi\fi\fi\fi\fi\fi\WinEdt{<-}\global\LASTLINE=\inputlineno\global\advance\LASTLINE by -1\SRC}
\def\INPUT#1{\relax}
\let\OldINCLUDE=\include
\def\include#1{\EJECT\PUSH{#1.tex}\OldINCLUDE{#1}\POP}

\let\originalxxxeverypar\everypar
\newtoks\everypar
\originalxxxeverypar{\the\everypar\expandafter\SRC}
\everymath\expandafter{\the\everymath\expandafter\SRC}
\let\zzzxxxbibliography=\bibliography
\def\bibliography#1{\PUSH{\jobname.bbl}\zzzxxxbibliography{#1}\POP}
\output\expandafter{\expandafter\SRCOKfalse\the\output}

\input{tcilatex}

\begin{document}
\title[Elliptic problems with dynamical boundary conditions]{Nonlinear
elliptic problems with dynamical boundary conditions of reactive and
reactive-diffusive type}
\author{Ciprian G. Gal}
\address{Department of Mathematics, Florida International University, Miami,
FL 33199, USA}
\email{cgal@fiu.edu}
\author{Martin Meyries}
\address{Department of Mathematics, Karlsruhe Institute of Technology, 76128
Karlsruhe, Germany}
\email{martin.meyries@kit.edu}
\thanks{The second author was supported by the project ME 3848/1-1 of
Deutsche Forschungsgemeinschaft (DFG)}
\maketitle

\begin{abstract}
We investigate classical solutions of nonlinear elliptic equations with two
classes of dynamical boundary conditions, of reactive and reactive-diffusive
type. In the latter case it is shown that well-posedness is to a large
extent independent of the coupling with the elliptic equation. For both
types of boundary conditions we consider blow-up, global existence, global
attractors and convergence to single equilibria.
\end{abstract}

\tableofcontents

\section{Introduction}

The prototype of the elliptic-parabolic initial-boundary value problems that
we consider in this article is 
\begin{equation}
\left\{ 
\begin{array}{ll}
\lambda u-d\Delta u=f(u) & \text{in }(0,T)\times \Omega , \\ 
\partial _{t}u_{\Gamma }-\delta \Delta _{\Gamma }u_{\Gamma }+d\partial _{\nu
}u=g(u_{\Gamma }) & \text{on }(0,T)\times \Gamma , \\ 
u|_{\Gamma }=u_{\Gamma } & \text{on }(0,T)\times \Gamma , \\ 
u_\Gamma|_{t=0}=u_{0} & \text{on }\Gamma .%
\end{array}%
\right.  \label{ell-dyn-intro}
\end{equation}%
We assume that $\Omega \subset \mathbb{R}^{n}$ is a bounded domain with
smooth boundary $\Gamma =\partial \Omega $, that $d>0$, $\delta \geq 0$ and $%
f,g\in C^{\infty }(\mathbb{R})$. Further, $\Delta _{\Gamma }$ is the
Laplace-Beltrami operator and $\partial _{\nu }$ is the outer normal
derivative on $\Gamma $. It is throughout assumed that $f$ is globally
Lipschitz continuous and that $\lambda $ is sufficiently large, in
dependence on $f$. Depending on $\delta $, two classes of boundary
conditions are modelled by \eqref{ell-dyn-intro}. For $\delta >0$ we have
boundary conditions of reactive-diffusive type, and for $\delta =0$ the
boundary conditions are purely reactive.

The motivation to consider (\ref{ell-dyn-intro}) comes from physics. The
function $u$ represents the steady state temperature in a body $\Omega $
such that the rate at which $u$ evolves through the boundary $\Gamma $ is
proportional to the flux on the boundary, up to some correction $\delta
\Delta _{\Gamma }u_\Gamma,$ $\delta \geq 0$, which from a modelling
viewpoint, accounts for small diffusive effects along $\Gamma$. Moreover,
the heat source on $\Gamma $ acts nonlinearly through the function $g$.
Problem (\ref{ell-dyn-intro}) is also important in conductivity (see, e.g., 
\cite{Gr})\ and harmonic analysis due to its connection to the following
eigenvalue problem%
\begin{equation}
\Delta u=0\quad \text{ in }\Omega ,\qquad -\delta \Delta _{\Gamma }u_\Gamma
+\partial _{\nu }u=\xi u_\Gamma \quad \text{ on }\Gamma ,  \label{Sev}
\end{equation}%
which was introduced by Stekloff \cite{St} (initially) in the case $\delta
=0 $. This connection arises because the linear problem associated with (\ref%
{ell-dyn-intro}) (i.e., by letting $\lambda =0$, $f\equiv 0$ and $g\equiv 0$%
) can be solved by the Fourier method in terms of the eigenfunctions of (\ref%
{Sev}) (see \cite{VV09}, which also includes the case $\delta >0$; cf. also 
\cite{Vi} for $\delta =0$). The solvability of the linear problem (assuming $%
\delta =0$) was also investigated by Hintermann \cite{Hi} by means of the
theory of pseudo-differential operators, and by Gr\"{o}ger \cite{Gr} and
Showalter \cite{Sh}, by applying the theory of maximal monotone operators in
the Hilbert-space setting (see, also, \cite{GGZ}). It turns out that this
connection is also essential for solvability of the nonlinear problem (\ref%
{ell-dyn-intro}).

The mathematical study of the prototype (\ref{ell-dyn-intro}) has a
long-standing history. In \cite{Li} J.-L. Lions considered the special case $%
\delta =\lambda =0$, $f\equiv 0$ and $g\left( s\right) =-\left\vert
s\right\vert ^{p}s,$ $p>0$. By standard compactness methods, he proved
existence and uniqueness of global solutions for initial datum $u_{0}\in
H^{1/2}\left( \Gamma \right) $ in this special case. Problem (\ref%
{ell-dyn-intro}) was investigated in the general case by Escher \cite%
{Escher92, Escher94} for nontrivial functions $f,g,$ by also treating
systems of elliptic equations, but always in the case $\delta =0$. His
papers deal with classical solvability and global existence for smooth
initial data. In particular, global existence of classical solutions was
shown assuming $f$ is globally Lipschitz and that $g\left( s\right) s\leq 0$%
, for all $s\in \mathbb{R}$. Constantin, Escher and Yin \cite{CE02, Y}
established, in the case $\delta =\lambda =0$ and $f\equiv 0,$ some natural
structural conditions for the function $g$ so that global existence of
classical solutions holds. Their approach is based on global existence
criteria for ODEs. Boundedness of the global solutions for (\ref%
{ell-dyn-intro}) was shown by Fila and Quittner \cite{FQ} in the case when $%
\delta =\lambda =0$, $f\equiv 0$ and $g$ is a superlinear \emph{subcritical }%
nonlinearity. They have also proved that blow-up in finite time occurs for (%
\ref{ell-dyn-intro}) if $g\left( s\right) =\left\vert s\right\vert
^{p-1}s-as,$ $p>1,$ $a\geq 0$ and if the initial datum $u_{0}$ is "large"
enough \cite[Section 3]{FQ}. Blow-up phenomena for smooth solutions of (\ref%
{ell-dyn-intro}), when $\delta =0$ and $f\equiv 0,$ was also observed by
Kirane \cite{Kir} under some general assumptions on $g$, i.e., when $g(s)>0$%
, for all $s\geq s_{0}$, and%
\begin{equation*}
\int_{s_{0}}^{\infty }\frac{d\xi }{g(\xi )}<\infty .
\end{equation*}%
A version of the problem (\ref{ell-dyn-intro}) for which the dynamic
boundary condition is replaced by%
\begin{equation*}
\left\vert \partial _{t}u_{\Gamma }\right\vert ^{m-1}\partial _{t}u_{\Gamma
}+d\partial _{\nu }u=\left\vert u_{\Gamma }\right\vert ^{p-1}u_{\Gamma }
\qquad \text{on }\Gamma \times \left( 0,T\right) ,
\end{equation*}%
for some $m\geq 1$ and $p\geq 1$ was investigated by Vitillaro \cite{Vi} for
initial data $u_{0}\in H^{1/2}\left( \Gamma \right) $ and $f\equiv 0$. He
mainly devotes his attention to proving the local and the global existence
as well as blow-up of solutions for $m\geq 1$, especially, in the nonlinear
case when $m\neq 1$. Finally, it is interesting to note that, in the case
when $f\neq 0$ but $f$ is \emph{not} globally Lipschitz, global
non-existence without blow-up and non-uniqueness phenomena for (\ref%
{ell-dyn-intro}) can occur (see \cite{FP99}).

All the papers quoted so far deal only with classical issues, such as global
existence, uniqueness and blow-up phenomena for (\ref{ell-dyn-intro}) when $%
\delta =0$. Concerning further regularity and longtime behavior of
solutions, as time goes to infinity, not much seems to be known. This seems
to be due to the fact that the gradient structure of (\ref{ell-dyn-intro})
has not been exploited before. This issue is intimately connected with a 
\emph{key} result on smoothness in $\mathbb{R}_{+}\times \overline{\Omega }$
of solutions for (\ref{ell-dyn-intro}) even when $f\neq 0$ (see Proposition %
\ref{classic-diff}), which is essential to the study of the asymptotic
behavior of the system, in terms of global attractors and $\omega $-limit
sets.

The main novelties of the present paper with respect to previous results on (%
\ref{ell-dyn-intro}) are the following:

(\textbf{i}) The local well-posedness results are extended to the case $%
\delta >0$. In fact, we will consider a more general class of elliptic
problems with quasilinear, nondegenerate dynamic boundary conditions of
reactive-diffusive type. More precisely, we consider the following
generalization of the prototype model \eqref{ell-dyn-intro},%
\begin{equation}
\left\{ 
\begin{array}{ll}
\lambda u+\mathcal{A}u=f(u) & \text{in }(0,T)\times \Omega , \\ 
\partial _{t}u_{\Gamma }+\mathcal{C}(u_{\Gamma })u_{\Gamma }+\mathcal{B}%
(u)=g(u_{\Gamma }) & \text{on }(0,T)\times \Gamma , \\ 
u_\Gamma|_{t=0}=u_{0} & \text{on }\Gamma ,%
\end{array}%
\right.  \label{general}
\end{equation}%
where%
\begin{equation*}
\mathcal{A}u=-\text{div}\big(d\nabla u\big),\qquad \mathcal{C}(u_{\Gamma
})u_{\Gamma }=-\text{div}_{\Gamma }\big(\delta (\cdot ,u_{\Gamma })\nabla
_{\Gamma }u_{\Gamma }\big),
\end{equation*}%
such that $d\in C^{\infty }(\overline{\Omega }),$ $\delta \in C^{\infty
}(\Gamma \times \mathbb{R})$ with $d\geq d_{\ast }>0$ and $\delta \geq
\delta _{\ast }>0$. Moreover, $\nabla_\Gamma$ is the surface gradient and $%
\text{div}_\Gamma$ is the surface divergence. Here and in the sequel we
always assume that $u|_\Gamma = u_\Gamma$. The nonlinear map $\mathcal{B}$
in (\ref{general}) couples the equations in the domain $\Omega $ and on the
boundary $\Gamma $ in a (possibly) nontrivial way. We do \emph{not} impose
any further structural conditions for $\mathcal{B}$ and $g$ other than they
must be of order strictly lower than two and satisfy a local Lipschitz
condition. One example for $\mathcal{B}$ we have in mind is $\mathcal{B}%
(u)=b \nu \cdot (\nabla u)|_{\Gamma }$, with \emph{no} sign restriction on $%
b\in C^{\infty }(\Gamma)$. We prove that for sufficiently large $\lambda $
and a \emph{globally} Lipschitz function $f$ the problem (\ref{general})
generates a (compact) local semiflow of solutions for $u_{0}\in \mathcal{X}%
_{\delta }:=W^{2-2/p,p}(\Gamma )$, $p\in \left( n+1,\infty \right) $, $%
\delta >0$, and establish some further regularity properties for the local
solution $u=u(\cdot ;u_{0})$. For the notion of local semiflow, we refer the
reader to Section \ref{semi}.

The independence of the well-posedness of the coupling was first observed by
Vazquez and Vitillaro \cite{VV09} for a linear model problem with $\mathcal{%
C }= -\Delta_\Gamma$ and $\mathcal{B}=-\partial _{\nu }$ in a Hilbert space
setting. Our approach to the quasilinear problem is based on maximal $L^p$%
-regularity properties of the corresponding linearized dynamic equation on
the boundary. In Section \ref{mbdo} these will be verified for a general
class of elliptic boundary differential operators using localization
techniques. The global Lipschitz condition on $f$ allows to solve the
elliptic equation on $\Omega$ and to rewrite (\ref{general}) as an
initial-value problem for $u_\Gamma$ on $\Gamma$, which can be treated with
the general theory of \cite{KPW10}. The fact that the concrete form of the
coupling $\mathcal{B}$ is inessential is a consequence of the fact that
maximal regularity is invariant under lower order perturbations. For the
precise statements of these results we refer the reader to Section \ref%
{wellp}.

The corresponding result for boundary conditions of purely reactive type,
i.e., $\mathcal{C}\equiv 0$ and $\mathcal{B}=d\partial _{\nu }$ in (\ref%
{general}), was shown in \cite{Escher92}. There the result is based on the
generation properties of the Dirichlet-Neumann operator and thus, the
solutions enjoy worse regularity properties up to $t=0$. In addition to this
we establish the compactness of the solution semiflow on $\mathcal{X}_0 :=
W^{1-1/p,p}(\Gamma)$, $p\in (n,\infty)$, in this case (see Section \ref{cprc}%
).

(\textbf{ii}) The blow-up results for problem (\ref{ell-dyn-intro}), from 
\cite{Kir} and \cite{Vulkov}, are also extended to the case when $\delta >0$
and $f\neq 0$. Our approach is based on the method of subsolutions and a
comparison lemma, and is inspired by \cite{AMTR} and \cite{Rothe} (see
Section \ref{bl}). We further show global existence of solutions of (\ref%
{ell-dyn-intro}) under the natural assumption that $g(\xi )\xi \leq
c_{g}(\left\vert \xi \right\vert ^{2}+1)$ for all $\xi \in \mathbb{R}$ by
performing a Moser-Alikakos iteration procedure as in \cite{Gal0, Mey10}.
Here an inequality of Poincar\'{e}-Young type allows to connect the
structure of the elliptic equation with that of the dynamic equation on $%
\Gamma $ (see Section \ref{globals}).

(\textbf{iii}) We prove the smoothness of solutions of (\ref{ell-dyn-intro}%
)\ in both space and time exploiting a variation of parameters formula for
the trace $u_{\Gamma }$ and the implicit function theorem, which is entirely
new (see Section \ref{classical-solutions}). Consequently, taking advantage
of this smoothness, we can show that (\ref{ell-dyn-intro}) has a gradient
structure, and as a result establish the existence of a finite-dimensional
global attractor in the phase space $\mathcal{X}_{\delta }$ for both types
of boundary conditions. Here the main assumption is that the first
eigenvalue of a Stekloff-like eigenvalue problem (similar to (\ref{Sev})) is
positive (see Section \ref{attractors}).


(\textbf{iv}) The $\omega $-limit sets of (\ref{ell-dyn-intro}) can exhibit
a complicated structure if the functions $f,g$ are non-monotone and, a
fortiori, the same is true for the global attractor. Indeed, when $f,g$ are
non-monotone (i.e., the related potentials $F\left( s\right)
=\int_{0}^{s}f\left( y\right) dy,$ $G\left( s\right) =\int_{0}^{s}g\left(
y\right) dy$ are non-convex) this can happen if the stationary problem
associated with (\ref{ell-dyn-intro})\ possesses a continuum of nonconstant
solutions. Some examples which show that the $\omega $-limit set can be a
continuum are provided in \cite{PS}. However, assuming the nonlinearities $%
f,g$ to be real analytic, we prove the convergence of a given trajectory $%
u=u(t;u_{0}),$ $u_{0}\in \mathcal{X}_{\delta },$ as time goes to infinity,
to a single equilibrium of (\ref{ell-dyn-intro}). This shows, in a strong
form, the asymptotic stability of $u(t;u_{0})$ for an arbitrary (but given)
initial datum $u_{0}\in \mathcal{X}_{\delta }$. This type of result exploits
a technique which is based on the so-called \L ojasiewicz-Simon inequality
(see Section \ref{cte}; cf. also \cite{SW, Wu}).

Finally, it is worth mentioning that most of our results can be also
extended to systems of nonlinear elliptic equations subject to both types of
boundary conditions.


The \emph{plan of the paper} goes as follows. In Section \ref{prelim}, we
introduce the functional analytic framework associated with (\ref%
{ell-dyn-intro})\ and (\ref{general}), respectively. In Section \ref{mbdo},
maximal $L^{p}$-regularity theory is developed for elliptic boundary
differential operators of second order. Then, in Section \ref{wellp} (and
corresponding subsections) we prove (local) well-posedness results for %
\eqref{general} and establish the existence of a compact (local) semiflow on
the corresponding phase spaces. The final Section \ref{qp} is further
divided into five parts: the first part provides the key result which shows
the smoothness of solutions in both space and time, while the second and
third parts deal with blow-up phenomena and global existence, respectively.
Finally, the last two subsections deal with the asymptotic behavior as time
goes to infinity, in terms of global attractors and convergence of solutions
to single equilbria.

\section{Preliminaries}

\label{prelim}

\subsection{Function spaces}

We briefly describe the function spaces that are used in the paper. Details
and proofs can be found in \cite{Lunardi09, Tri94}.

Throughout, all function spaces under consideration are real. Let $p\in
\lbrack 1,\infty ]$. If $\Omega \subseteq \mathbb{R}^{n}$ is open, we denote
by $L^{p}(\Omega )$ the usual Lebesgue spaces. Now let $\Omega $ have a
(sufficiently) smooth boundary. Then for $s\geq 0$ and $p\in \lbrack
1,\infty )$ we denote by $H^{s,p}(\Omega )$ the Bessel-potential spaces and
by $W^{s,p}(\Omega )$ the Slobodetskij spaces. One has $H^{s,2}(\Omega
)=W^{s,2}(\Omega )$ for all $s$, but for $p\neq 2$ the identity $%
H^{s,p}(\Omega )=W^{s,p}(\Omega )$ is only true if $s\in \mathbb{N}_{0}$. If 
$s\in \mathbb{N}_{0}$, then $H^{s,p}(\Omega )$ and $W^{s,p}(\Omega )$
coincide with the usual Sobolev spaces. In the case of noninteger
differentiability, for our purposes it suffices to consider these spaces as
interpolation spaces. If $s=[s]+s_{\ast }\notin \mathbb{N}_{0}$ with $[s]\in 
\mathbb{N}_{0}$ and $s_{\ast }\in (0,1)$, then 
\begin{equation}
H^{s,p}=[H^{[s],p},H^{[s]+1,p}]_{s_{\ast }},\qquad
W^{s,p}=(W^{[s],p},W^{[s]+1,p})_{s_{\ast },p},  \label{interpb}
\end{equation}%
where $[\cdot ,\cdot ]_{s_{\ast }}$ and $(\cdot ,\cdot )_{s_{\ast },p}$
denote complex and real interpolation, respectively. Moreover, $%
H^{s,p}=[L^{p},H^{2,p}]_{s/2}$ for $s\in (0,2)$ and $W^{s,p}(\Omega
)=(L^{p},W^{2,p})_{s/2,p}$ for $s\in (0,2)$, $s\neq 1$. A useful tool are
interpolation inequalities. We shall make particular use of 
\begin{equation}
\Vert u\Vert _{H^{s,p}}\leq \Vert u\Vert _{L^{p}}^{1-s/2}\Vert u\Vert
_{H^{2,p}}^{s/2},\qquad \Vert u\Vert _{W^{s,p}}\leq C\,\Vert u\Vert
_{L^{p}}^{1-s/2}\Vert u\Vert _{W^{2,p}}^{s/2},  \label{interp}
\end{equation}%
which is valid for all $u\in H^{2,p}=W^{2,p}$.

The corresponding function spaces over the boundary $\Gamma = \partial\Omega$
of a bounded smooth domain $\Omega \subset \mathbb{R}^n$ are defined via
local charts. Let $\text{g}_i:U_i\subset \mathbb{R}^{n-1} \to \Gamma$ be a
finite family of parametrizations such that $\bigcup_i \text{g}_i(U_i)$
covers $\Gamma$, and let $\{\psi_i\}$ be a partition of unity for $\Gamma$
subordinate to this cover. Then for $s\geq 0$ we have 
\begin{equation*}
H^{s,p}(\Gamma) = \big \{ u \in L^p(\Gamma)\;:\; (\psi_i u)\circ \text{g}_i
\in H^{s,p}(\mathbb{R}^{n-1}) \text{ for all $i$}\big\},
\end{equation*}
and an equivalent norm is given by $\|u\|_{H^{s,p}(\Gamma)} = \sum_i
\|(\psi_i u)\circ \text{g}_i\|_{H^{s,p}(\mathbb{R}^{n-1})}. $ The spaces $%
W^{s,p}(\Gamma)$ are defined in the same way, replacing $H$ by $W $. In this
way the properties of the spaces over $\Omega$ described above carry over to
the spaces over $\Gamma$.

For $p\in (1,\infty)$ and $s > 1/p$ the trace $\text{tr} \,u = u|_\Gamma$
extends to a continuous operator 
\begin{equation*}
\text{tr}: H^{s,p}(\Omega) \to W^{s-1/p,p}(\Gamma).
\end{equation*}
Here we exclude the case $s-1/p \in \mathbb{N}$ for $p\neq 2$.


\subsection{Semiflows}

\label{semi}

Let $\mathcal{X}$ be a Banach space and let $t^+: \mathcal{X }\to (0,\infty]$
be lower semicontinuous. Then we call a map 
\begin{equation*}
S: \bigcup_{x\in \mathcal{X}} [0,t^+(x)) \times \{x\} \to \mathcal{X}
\end{equation*}
a local semiflow on $\mathcal{X}$ if for all $x\in \mathcal{X}$ it holds
that $S(\cdot;x):[0,t^+(x)) \to \mathcal{X}$ is continuous, if $S(t,\cdot):
B_r(x) \subset \mathcal{X }\to \mathcal{X}$ is continuous for $t < t^+(x)$
and sufficiently small $r >0$, if $S(0;\cdot ) = \text{id}_{\mathcal{X}}$, $%
S(t+s;x) = S(t; S(s;x))$ and if $t^+(x) < \infty$ implies that $\|S(t;x)\|_{%
\mathcal{X}} \to \infty$ as $t\to t^+$. In addition we call $S$ compact, if
for all bounded sets $M\subset \mathcal{X}$ with $t^+(M) \geq T > 0$ and all 
$t\in (0,T)$ it holds that $S(t;M)$ is relatively compact in $\mathcal{X}$.

If $t^+(x) = \infty$ for all $x\in \mathcal{X}$, then we call $S$ a global
semiflow. In this case our notion of a semiflow coincides with the one in 
\cite{CD}.

Note that, in contrast to parts of the literature, we include the condition
for global existence (i.e., $t^+ = \infty$) already in the definition of a
local semiflow.

\section{Maximal $L^{p}$-regularity for boundary differential operators}

\label{mbdo}

In this section we show maximal $L^p$-regularity for elliptic boundary
differential operators of second order.

\subsection{Boundary differential operators}

Throughout, let $\Omega \subset \mathbb{R}^{n}$ be a bounded domain with
smooth boundary $\Gamma =\partial \Omega$. We describe our notion of a
differential operator on $\Gamma$ with possibly nonsmooth coefficients.

Let $(0,T)$ be a finite or infinite time interval. We call a globally
defined, linear map $\mathcal{C}:(0,T)\times C^{\infty }(\Gamma )\rightarrow
L^{1}(\Gamma )$ a (non-autonomous) \emph{boundary differential operator} of
order $k\in \mathbb{N}$, if for all $t\in (0,T)$ and all parametrizations $%
\text{g}:U\subset \mathbb{R}^{n-1}\rightarrow \Gamma $ it holds 
\begin{equation*}
(\mathcal{C}(t,\cdot )u)\circ \text{g}(x)=\sum_{|\gamma |\leq k}c_{\gamma }^{%
\text{g}}(t,x)D_{n-1}^{\gamma }(u\circ \text{g})(x),\qquad x\in U,
\end{equation*}%
with local coefficients $c_{\gamma }^{\text{g}}(t,\cdot )\in L^{1}(U)$ and $%
D_{n-1}=-\text{i}\nabla _{n-1}$. The coefficients do not have to be globally
defined and may in fact depend on the parametrization $\text{g}$. The
examples we have in mind are the Laplace-Beltrami operator $\Delta _{\Gamma
}=\text{div}_{\Gamma }\nabla _{\Gamma }$, which is in coordinates given by 
\begin{equation*}
(\Delta _{\Gamma }u)\circ \text{g}=\frac{1}{\sqrt{\left\vert \text{G}%
\right\vert }}\sum_{i,j=1}^{n-1}\partial _{i}(\sqrt{|\text{G}|}\text{g}%
^{ij}\partial _{j}(u\circ \text{g})),
\end{equation*}%
and, for a tangential vector field $\mathcal{V}$ on $\Gamma $, a surface
convection term $\mathcal{V}\nabla _{\Gamma }$, i.e., 
\begin{equation*}
(\mathcal{V}\nabla _{\Gamma }u)\circ \text{g}=\sum_{i,j=1}^{n-1}\text{g}%
^{ij}(\mathcal{V}\cdot \partial _{i}\text{g})\partial _{j}(u\circ \text{g}).
\end{equation*}%
Here $\text{G}^{-1}=(\text{g}^{ij})_{i,j}$ is the inverse of the fundamental
form $\text{G}$ corresponding to $\text{g}$.

As in the euclidian case, the regularity of the local coefficients $%
c_\gamma^{\text{g}}$ decides on which scale of function spaces over $\Gamma $
the operator $\mathcal{C}(t,\cdot )$ acts. For instance, if $c_{\gamma }^{%
\text{g}}(t,\cdot )\in L^{\infty }(U)$ for all parametrizations $g$ and all $%
\gamma$, then we obtain for all $p\in \lbrack 1,\infty ]$ an estimate 
\begin{equation*}
\Vert \mathcal{C}(t,\cdot )u\Vert _{L^{p}(\Gamma )}\leq C\,\Vert u\Vert
_{W^{k,p}(\Gamma )},\qquad u\in C^{\infty }(\Gamma ).
\end{equation*}%
In this case $\mathcal{C}(t,\cdot )$ extends uniquely to a bounded linear
map $W^{k,p}(\Gamma )\rightarrow L^{p}(\Gamma )$, or to a closed operator on 
$L^{p}(\Gamma )$ with domain $W^{k,p}(\Gamma )$. Of course, in view of
Sobolev embeddings, for such an extension the regularity of the coefficients
can be lowered in many cases.

Finally, structural conditions like ellipticity of a boundary differential
operator $\mathcal{C}$ can also be imposed to hold locally with respect to
all parametrizations, see e.g. condition (E) below.

\subsection{Maximal $L^p$-regularity}

Let $\mathcal{C}$ be a boundary differential operator of order $k=2$.
Consider for a finite time interval $(0,T)$ the inhomogeneous Cauchy problem 
\begin{equation*}
\partial_{t}u + \mathcal{C}(t,x) u=g(t,x)\quad \text{on }(0,T)\times \Gamma
,\qquad u|_{t=0}=u_{0}\quad \text{on }\Gamma .
\end{equation*}
For $p\in (1,\infty)$ we take $g\in L^p ((0,T)\times \Gamma)$ and are thus
looking for solutions $u$ that belong to the space 
\begin{equation*}
\mathbb{E}(\Gamma) := W^{1,p}(0,T; L^p(\Gamma)) \cap L^p(0,T;
W^{2,p}(\Gamma)).
\end{equation*}
We want that for all parametrizations $\text{g}:U\to \Gamma$ and all $%
|\gamma|\leq 2$ the terms $c_\gamma^\text{g} D_{n-1}^\gamma$ arising in the
local representation of $\mathcal{C}$ are continuous from $\mathbb{E}(U)$ to 
$L^p((0,T)\times U)$. Then in particular $\mathcal{C}:\mathbb{E}(\Gamma) \to
L^p((0,T)\times \Gamma)$ is continuous. Using Sobolev embeddings and
H\"older's inequality, it can be shown as in \cite[Lemma 1.3.15]{Mey10} that
the following assumptions are sufficient for this purpose.

\begin{itemize}
\item[(R)] Let $\text{g}:U\rightarrow \Gamma$ be any parametrization of $%
\Gamma $. Then for $|\gamma |=2$ it holds $c_{\gamma }^{\text{g}}\in
BU\!C([0,T]\times U)$, and in case $|\gamma |<2$ one of the following
conditions is valid: either $p > n+1$ and $c_{\gamma }^{\text{g}}\in
L^{p}((0,T)\times U)$, or there are $r_{\gamma },s_{\gamma }\in \lbrack
p,\infty )$ with $\frac{1}{s_{\gamma }}+\frac{n-1}{2r_{\gamma }}<1-\frac{%
|\gamma |}{2}$ such that $c_{\gamma }^{\text{g}}\in L^{s_{\gamma
}}(0,T;L^{r_{\gamma }}(U))$.
\end{itemize}

As structural conditions for $\mathcal{C}$ we assume \emph{local
parameter-ellipticity} (cf. \cite{DHP03} for the euclidian case). Observe
that this is a condition only for the highest order coefficients.

\begin{itemize}
\item[(E)] For all parametrizations $g:U\rightarrow \Gamma $, all $t\in
\lbrack 0,T]$, $x\in U$ and $\xi \in \mathbb{R}^{n-1}$ with $|\xi |=1$ it
holds that $\sum_{|\gamma |=2}c_{\gamma }^{\text{g}}(t,x)\xi ^{\gamma } > 0$.
\end{itemize}

\begin{example}
Let $\mathcal{C }u = - \text{div}_\Gamma (\delta \nabla_\Gamma u)$ for a $%
C^1 $-function $\delta: (0,T)\times \Gamma \to \mathbb{R}$ with $\delta \geq
\delta_* > 0$. Then $\mathcal{C}$ satisfies (E).
\end{example}

Our maximal $L^p$-regularity result is now as follows.

\begin{theorem}
\label{thm1}Let $p\in (1,\infty )$ and $T\in (0,\infty )$. Assume that $%
\Omega\subset \mathbb{R}^n$ is a bounded domain with smooth boundary $\Gamma
= \partial\Omega$, and that $\mathcal{C}$ is a differential operator on $%
\Gamma$ of second order satisfying \emph{(R)} and \emph{(E)}. Then there is
a unique solution 
\begin{equation*}
u\in \mathbb{E}(\Gamma)=W^{1,p} (0,T;L^{p}(\Gamma))\cap L^{p}
(0,T;W^{2,p}(\Gamma))
\end{equation*}%
of the problem 
\begin{equation}
\left\{ 
\begin{array}{ll}
\partial _{t}u+\mathcal{C}(t,x)u=g(t,x) & \emph{\text{on }} (0,T)\times
\Gamma, \\ 
u|_{t=0}=u_{0} & \emph{\text{on }}\Gamma ,%
\end{array}%
\right.  \label{1}
\end{equation}%
if and only if the data is subject to 
\begin{equation*}
g\in L^{p}((0,T)\times \Gamma), \qquad u_{0}\in W^{2-2/p,p}(\Gamma).
\end{equation*}
Given $T_{0}>0$, there is a constant $C$, which is independent of the data
and $T\in (0,T_{0})$, such that 
\begin{equation}
\Vert u\Vert _{\mathbb{E}(\Gamma )}\leq C\left( \Vert g\Vert
_{L^{p}((0,T)\times \Gamma)}+\Vert u_{0}\Vert _{W^{2-2/p,p}(\Gamma )}\right)
.  \label{22222}
\end{equation}
Moreover, in the autonomous case, i.e., if $\mathcal{C}$ is independent of $%
t $, then $- \mathcal{C}$ generates an analytic $C_0$-semigroup on $%
L^p(\Gamma) $.
\end{theorem}

\begin{proof}
\emph{Step 1.} If $u\in \mathbb{E}(\Gamma)$ solves (\ref{1}), then $g\in
L^{p}((0,T)\times \Gamma)$ follows from (R), and further $u_{0}\in
W^{2-2/p,p}(\Gamma )$ is a consequence of e.g. \cite[Theorem 4.2]{MS11a}.
Next, assume that a unique solution of \eqref{1} exists for all given data.
Then the corresponding solution operator is continuous $L^{p}((0,T)\times
\Gamma)\times W^{2 -2/p,p}(\Gamma )\rightarrow \mathbb{E}(\Gamma)$ due to
the open mapping theorem. This gives \eqref{22222}. The uniformity of the
constant with respect to $T\in (0,T_0)$ follows from an
extension-restriction argument and the uniqueness of solutions. Further, in
this case the generator property of $- \mathcal{C}$ follows from \cite[%
Corollary 4.4]{Do00}, and the strong continuity of the semigroup is a
consequence of the density of $W^{2,p}(\Gamma)$ in $L^p(\Gamma)$.

We thus have to show the unique solvability of (\ref{1}) in $\mathbb{E}%
(\Gamma)$ for all data $g\in L^p((0,T)\times \Gamma)$ and $u_0\in
W^{2-2/p,p}(\Gamma) $. A compactness argument shows that it suffices to do
this for one (possibly small) $T>0$, which is independent of the data.

\emph{Step 2.} Choose a finite number of parametrizations $\text{g}_{i}$
with domains $U_{i}$ such that $\bigcup_i \text{g}_i(U_i)$ covers $\Gamma$,
and a partition of unity $\{\psi _{i}\}$ for $\Gamma $ subordinate to this
cover. Then $u\in \mathbb{E}(\Gamma)$ solves \eqref{1} if and only if for
all $i,$ the function $v_{i}:=(\psi _{i}u)\circ \text{g}_{i}$ solves%
\begin{equation}
\partial _{t}v_{i}+\mathcal{C}^{\text{g}_{i}}v_{i}=g_{i}, \quad \text{on }%
(0,T)\times U_{i}, \qquad v_{i}|_{t=0}=v_{i}^{0}, \quad \text{on }U_{i}.
\label{6}
\end{equation}%
Here the local operator $\mathcal{C}^{\text{g}_{i}}$ is given by%
\begin{equation*}
\mathcal{C}^{\text{g}_{i}}(t,x)=\sum_{|\gamma |\leq 2}c_{\gamma }^{\text{g}%
_{i}}(t,x)D^{\gamma },\qquad (t,x)\in (0,T)\times U_{i},
\end{equation*}%
and the transformed data is given by%
\begin{equation*}
g_{i}:=\left( \psi _{i}g+[\mathcal{C},\psi _{i}]u\right) \circ \text{g}%
_{i},\qquad v_{i}^{0}:=(\psi _{i}u^{0})\circ \text{g}_{i},
\end{equation*}%
where $[\cdot ,\cdot ]$ denotes the commutator bracket, i.e., $[\mathcal{C}%
,\psi _{i}]u = \mathcal{C}(\psi_i u) - \psi_i \mathcal{C }u.$

Identifying $v_{i}$, $g_{i}$ and $v_{i}^{0}$ with their trivial extensions
to $\mathbb{R}^{n-1}$, we obtain 
\begin{equation*}
v_{i}\in \mathbb{E}(\mathbb{R}^{n-1}), \qquad g_{i}\in L ^{p}((0,T)\times 
\mathbb{R}^{n-1}), \qquad v_{i}^{0}\in W^{2 -2/p,p}(\mathbb{R}^{n-1}).
\end{equation*}
We extend the top order coefficients $c_{\gamma }^{\text{g}_{i}} $, $|\gamma
|=2$, to coefficients $c_{\gamma }^{i}\in BU\!C([0,T]\times \mathbb{R}%
^{n-1}) $, using only values from the image of $c_{\gamma }^{\text{g}_{i}}$
(assuming e.g. $U_{i}$ to be ball and reflecting on $\partial U_i$). By
continuity of the $c_{\gamma }^{\text{g}_{i}} $, the oscillation of the
extend top order coefficients becomes small if the diameter of the $U_{i}$
is small. The lower order coefficients $c_{\gamma }^{\text{g}_{i}}$, $%
|\gamma |<2$, are trivially extended to $c_{\gamma }^{i}$ on $(0,T)\times 
\mathbb{R}^{n-1}$. The extended coefficients $c_\gamma^i$, $|\gamma|\leq 2$,
induce a differential operator $\mathcal{C}^{i}(t,x)$ acting on functions
over $(0,T)\times \mathbb{R}^{n-1}$, which satisfies (R) and (E) for $\Gamma
= \mathbb{R}^{n-1}$. Note in particular that (E) is a pointwise condition.
The (trivial extension of) $v_{i}$ solves the full space problem 
\begin{equation}
\partial _{t}w+\mathcal{C}^{i}(t,x)w=g_{i}\quad \text{on }(0,T)\times 
\mathbb{R}^{n-1}, \qquad w|_{t=0}=v_{i}^{0} \quad \text{on }\mathbb{R}^{n-1}.
\label{3}
\end{equation}%
It is well-known that there is a solution operator $\mathcal{S}_i(g_i,
v_0^i) $ for (\ref{3}), which is continuous 
\begin{equation*}
\mathcal{S}_{i}:L^{p}((0,T)\times \mathbb{R}^{n-1})\times W^{2 -2/p,p}(%
\mathbb{R}^{n-1})\rightarrow \mathbb{E}(\mathbb{R}^{n-1}).
\end{equation*}%
We refer to \cite[Proposition 2.3.2]{Mey10} for the case of top order
coefficients with small oscillation, which applies to the present case.
Hence 
\begin{equation*}
v_{i}=\mathcal{S}_{i}(g_{i},v_{i}^{0})|_{U_{i}}.
\end{equation*}

\emph{Step 3.} Take $\phi _{i}\in C^\infty(\Gamma)$ with $\phi _{i}\equiv 1$
on $\text{supp}\,\psi _{i}$ and $\text{supp}\,\phi _{i}\subset \text{g}%
_i(U_{i})$. On the complete metric space 
\begin{equation*}
\mathbb{Y}_{u_0}:=\left\{ u\in \mathbb{E}(\Gamma)\;:\;
u|_{t=0}=u_{0}\right\} ,
\end{equation*}%
which is nonempty by \cite[Lemma 4.3]{MS11a}, we define the map $\mathcal{S}%
_{g,u_{0}}$ by%
\begin{equation*}
\mathcal{S}_{g,u_{0}}(u):=\sum_{i}\phi _{i}\mathcal{S}_{i}\big( (\psi _{i}g+[%
\mathcal{C},\psi _{i}]u)\circ \text{g}_{i},(\psi _{i}u_{0})\circ \text{g}_{i}%
\big) |_{U_{i}}\circ \text{g}_{i}^{-1}.
\end{equation*}%
Since $\mathcal{S}_i(g_i,v_i^0)|_{t=0}= v_i^0$, we have that $\mathcal{S}%
_{g,u^{0}}$ is a self-mapping on $\mathbb{Y}_{u_0}$. We show that it is a
strict contraction if $T$ is sufficiently small. For $u,v\in \mathbb{Y}%
_{u_0} $ we have 
\begin{align}
\Vert \mathcal{S}_{g,u^{0}}(u)-\mathcal{S}_{g,u^{0}}(v)\Vert _{\mathbb{E}%
(\Gamma )}& \,\leq C\sum_i \Vert \mathcal{S}_{i}\left( [\mathcal{C},\psi
_{i}](u-v)\circ g_{i},0\right) \Vert _{\mathbb{E}(\mathbb{R}^{n-1})}  \notag
\\
& \,\leq C \sum_{i} \Vert \lbrack \mathcal{C},\psi _{i}](u-v)\circ
g_{i}\Vert _{L^{p}((0,T)\times \mathbb{R}^{n-1})}.  \label{222}
\end{align}%
Since $[\mathcal{C},\psi _{i}](u-v)\circ \text{g}_{i}$ involves at most the
first derivatives of $(u-v)\circ \text{g}_i$, we have for arbitrary $\eta >0$
that 
\begin{align}
\Vert \lbrack \mathcal{C},\psi _{i}](u-v)\circ \text{g}_{i}\Vert
_{L^{p}((0,T)\times \mathbb{R}^{n-1})}& \,\leq C\,\Vert (u-v)\circ \text{g}%
_i\Vert _{L^{p}(0,T;W^{1,p}(\mathbb{R}^{n-1}))}  \notag \\
& \,\leq \eta \Vert u-v\Vert _{L^{p}(0,T;W^{2,p}(\Gamma ))}  \notag \\
& \quad +C_{\eta }\Vert u-v\Vert _{L^{p}((0,T)\times \mathbb{R}^{n-1})} 
\notag \\
& \,\leq (\eta +C_{\eta }T)\Vert u-v\Vert _{\mathbb{E}(\Gamma)},  \label{4}
\end{align}%
using the interpolation inequality for $W^{1,p}(\Gamma)$, Young's inequality
and Poincare's inequality for $L^p (0,T; L^p(\Gamma))$ (see \cite[Lemma 2.12]%
{MS11a}).

Thus if $\eta $ and $T$ are sufficiently small, then $\mathcal{S}_{g,u_{0}}$
has a unique fixed point on $\mathbb{Y}_{u_0}$. Observe that this is true
for all $g$ and $u_0$, and that the choice of $T$ is independent of $g$ and $%
u_0$. The considerations in Step 2 show that every solution of (\ref{1}) is
necessarily a fixed point of $\mathcal{S}_{g,u_{0}}$. We have thus already
shown that solutions of (\ref{1}) are unique. However, due to the nonempty
intersections of the $\text{g}_i(U_{i})$, the fixed point does in general
not solve (\ref{1}).

\emph{Step 4.} To find $g^{\ast }\in L^{p}((0,T)\times \Gamma)$ for which $%
\mathcal{S}_{g^{\ast },u_{0}}$ solves (\ref{1}) for given $g$ and $u_0 $ we
consider the fixed point map $\mathcal{F}$, defined by 
\begin{equation*}
\mathcal{F}:L^{p}((0,T)\times \Gamma)\times W^{2 -2/p,p}(\Gamma )\rightarrow 
\mathbb{E}(\Gamma),\qquad \mathcal{S}_{h,v_{0}}\left( \mathcal{F}%
(h,v_0)\right) =\mathcal{F}(h,v_0).
\end{equation*}%
For $h\in L^{p}((0,T)\times \Gamma)$ and $u_0 \in W^{2 -2/p,p}(\Gamma)$ we
have $\mathcal{F}(h,u_0)|_{t=0} = u_0$ and 
\begin{equation}
\left( \partial _{t}+\mathcal{C}\right) \mathcal{F}(h,u_{0})=h+\mathcal{K}h,
\label{8}
\end{equation}%
with the error term%
\begin{equation*}
\mathcal{K}h:=\sum_{i}[\mathcal{C},\phi _{i}]\mathcal{S}_{i}\big( (\psi
_{i}h+[\mathcal{C},\psi _{i}]\mathcal{F}(h,u_{0}))\circ \text{g}_{i},(\psi
_{i}u_{0})\circ \text{g}_{i}\big)|_{U_i} \circ \text{g}_{i}^{-1}.
\end{equation*}%
We use again the contraction principle to show that the map $h\mapsto g-%
\mathcal{K}h$ has a fixed point $g^{\ast }$ on $L^{p}((0,T)\times \Gamma)$.
Then $\mathcal{F}(g^{\ast },u_{0})$ solves (\ref{1}) for given $g\in
L^{p}((0,T)\times \Gamma)$ by (\ref{8}).

First note that $\mathcal{K}$ maps $L^{p}((0,T)\times \Gamma)$ into itself
by construction. For $h_{1},h_{2}\in L^{p}((0,T)\times \Gamma)$ we have that 
$\mathcal{F}(h_1,v_0)-\mathcal{F}(h_2,v_0) = \mathcal{F}(h_1-h_2,0)$, since
this difference is the unique solution of (\ref{1}) with inhomogeneity $%
h_1-h_2$ and trivial initial value. We thus obtain as above that 
\begin{align}
\Vert \mathcal{K}h_{1}-&\, \mathcal{K}h_{2}\Vert _{L^{p}((0,T)\times \Gamma)}
\notag \\
&\,\leq C\sum_{i}\Vert \mathcal{S}_{i}\big( (\psi _{i}(h_1-h_2)+[\mathcal{C}%
,\psi _{i}]\mathcal{F}(h_{1}-h_{2},0))\circ \text{g}_{i},0\big) \Vert
_{L^{p}(0,T;W^{1,p}(\mathbb{R}^{n-1} ))}  \notag \\
& \,\leq (\eta +C_{\eta }T)\sum_{i}\Vert \mathcal{S}_{i}\big( (\psi
_{i}(h_{1}-h_{2})+[\mathcal{C},\psi _{i}]\mathcal{F}(h_{1}-h_{2},0))\circ 
\text{g}_{i},0\big) \Vert _{\mathbb{E}(\Gamma )}  \notag \\
& \,\leq C(\eta +C_{\eta }T)\left( \Vert h_{1}-h_{2}\Vert
_{L^{p}((0,T)\times \Gamma)}+\Vert \lbrack \mathcal{C},\psi _{i}]\mathcal{F}%
(h_{1}-h_{2},0)\Vert _{L^{p}((0,T)\times \Gamma)}\right),  \notag
\end{align}%
where $\eta>0$ is arbitrary. Therefore $\mathcal{K}$ is a strict contraction
for sufficiently small $\eta$ and $T$ if the second summand above satisfies 
\begin{equation}  \label{2230}
\Vert \lbrack \mathcal{C},\psi _{i}]\mathcal{F}(h_{1}-h_{2},0)\Vert
_{L^{p}((0,T)\times \Gamma)} \leq C\, \|h_{1}-h_{2}\|_{L^{p}((0,T)\times
\Gamma)},
\end{equation}
with a constant $C$ independent of $T$. To see this we estimate for $h\in
L^{p}((0,T)\times \Gamma)$ 
\begin{align*}
\Vert \mathcal{F}(h,0)\Vert _{\mathbb{E}(\Gamma)} &\, = \Vert \mathcal{S}%
_{h,0}(\mathcal{F}(h,0))\Vert _{\mathbb{E}(\Gamma)} \\
& \,\leq \Vert \mathcal{S}_{h,0}(\mathcal{F}(h,0))-\mathcal{S}_{h,0}(0)\Vert
_{\mathbb{E}(\Gamma)}+\Vert \mathcal{S}_{h,0}(0)\Vert _{\mathbb{E}(\Gamma)}
\\
& \,\leq (\varepsilon +C_{\varepsilon }T)\Vert \mathcal{F}(h,0)\Vert _{%
\mathbb{E}(\Gamma)}+C\Vert h\Vert _{L^{p}((0,T)\times \Gamma)}
\end{align*}%
for given $\varepsilon>0$ by \eqref{222} and (\ref{4}). In this inequality,
if $\varepsilon$ and $T$ are sufficiently small, then we may absorb $%
(\varepsilon +C_{\varepsilon }T)\Vert \mathcal{F}(h,0)\Vert _{\mathbb{E}%
(\Gamma)}$ into the left-hand side to obtain 
\begin{equation*}
\Vert \mathcal{F}(h,0)\Vert _{\mathbb{E}(\Gamma)}\leq C\Vert
h\Vert_{L^{p}((0,T)\times \Gamma)}.
\end{equation*}
Now \eqref{2230} follows from 
\begin{equation*}
\Vert \lbrack \mathcal{C},\psi _{i}]\mathcal{F}(h_{1}-h_{2},0)\Vert
_{L^{p}((0,T)\times \Gamma)} \leq C\, \Vert \mathcal{F}(h_{1}-h_{2},0)\Vert
_{\mathbb{E}(\Gamma)},
\end{equation*}
which finishes the proof.
\end{proof}

\section{Well-posedness and compactness of the solution semiflows}

\label{wellp}

\subsection{Dirichlet problems}

Let $\Omega \subset \mathbb{R}^{n}$ be a bounded domain with smooth boundary 
$\Gamma $. We assume that the operator $\mathcal{A}$ is given by 
\begin{equation}  \label{opA}
\mathcal{A }u = - \text{div} \big( d \nabla u\big), \quad d\in C^\infty(%
\overline{\Omega},\mathbb{R}), \quad d\geq d_* >0,
\end{equation}
where $d_*$ is a constant. We first consider the linear inhomogeneous
Dirichlet problem 
\begin{equation}
\left\{ 
\begin{array}{ll}
\lambda v + \mathcal{A }v= f & \text{in }\Omega , \\ 
v|_\Gamma = g & \text{on }\Gamma.%
\end{array}%
\right.  \label{Diri}
\end{equation}
We denote $\text{tr}\, u = u|_\Gamma$ for the trace on $\Gamma$. It follows
from classical Agmon-Douglis-Nirenberg theory that there is $\lambda_D\geq 0$
such that for all $\lambda \geq \lambda_D$ it holds that 
\begin{equation*}
(\lambda+\mathcal{A}, \text{tr}): H^{2,p}(\Omega)\to L^p(\Omega)\times
W^{2-1/p,p}(\Gamma)
\end{equation*}
is a continuous isomorphism. For instance, if $\mathcal{A}=-d\Delta $ for a
constant $d>0$, then one can take $\lambda _{D}=0$. The corresponding
inverse, which is denoted by 
\begin{equation*}
\mathcal{R}_\lambda := (\lambda+\mathcal{A}, \text{tr})^{-1},
\end{equation*}
enjoys the following properties.

\begin{lemma}
\label{Diri-lemma} Let $\mathcal{A}$ be given by \eqref{opA} and $p\in
(1,\infty ).$ Then for all $\theta \in (1/2p,1]$, $\theta \neq 1/2 + 1/2p,$
and $\lambda \geq \lambda_D$ the operator $\mathcal{R}_\lambda$ extends to a
continuous map 
\begin{equation*}
\mathcal{R}_\lambda:L^{p}(\Omega )\times W^{2\theta -1/p,p}(\Gamma ) \to
H^{2\theta ,p}(\Omega).
\end{equation*}
There are constants $C_{D}$ (independent of $\lambda $) and $C_{\lambda }$
such that 
\begin{equation*}
\Vert \mathcal{R}_{\lambda }(f,g)\Vert _{H^{2\theta ,p}(\Omega )}\leq
C_{D}\lambda ^{-(1-\theta )}\Vert f\Vert _{L^{p}(\Omega )}+C_{\lambda }\Vert
g \Vert _{W^{2\theta -1/p,p}(\Gamma)}.
\end{equation*}
\end{lemma}

Here we exclude $\theta =1/2+1/2p$ to avoid the terminology of Besov spaces.

\begin{proof}
The properties of the extension of $\mathcal{R}_\lambda$ are proved in \cite[%
Eq. (12)]{Escher94} and are based on the Lions-Magenes extension of $%
(\lambda+\mathcal{A}, \text{tr})$, see \cite{LM}. Moreover, it is shown in 
\cite[Theorem 12.2]{Amann84} that 
\begin{equation*}
\|\mathcal{R}_\lambda(f,0)\|_{H^{2,p}(\Omega)} \leq C \|f\|_{L^p(\Omega)},
\qquad \|\mathcal{R}_\lambda(f,0)\|_{L^p(\Omega)} \leq C
\lambda^{-1}\|f\|_{L^p(\Omega)},
\end{equation*}
for all $\lambda\geq \lambda_D$. Now the second statement follows from
complex interpolation.
\end{proof}

Let us now consider the nonlinear Dirichlet problem 
\begin{equation}
\left\{ 
\begin{array}{ll}
\lambda u + \mathcal{A }u= F(u) & \text{in }\Omega , \\ 
u|_\Gamma = u_\Gamma & \text{on }\Gamma,%
\end{array}%
\right.  \label{Diri-nonlinear}
\end{equation}
where the boundary data $u_\Gamma$ is given. For $p\in (1,\infty)$ and $%
\theta \in (0,1)$ we assume that 
\begin{equation}
F:H^{2\theta,p}(\Omega )\rightarrow L^{p}(\Omega )\text{ is globally
Lipschitzian with constant $c_F \geq 0$.}  \label{f}
\end{equation}

\begin{example}
Let $F$ be the superposition operator induced by a globally Lipschitzian $f:%
\mathbb{R}\times \mathbb{R}^n \to \mathbb{R}$, i.e., $F(u)(x) =
f(u(x),\nabla u(x))$. Then $F$ satisfies \eqref{f} for all $p$ and $%
\theta\geq \frac{1}{2}$ . If $F$ is induced by a function that is not
globally Lipschitzian, then \eqref{f} cannot hold.
\end{example}

\begin{definition}
\label{sol_ep} Let $p\in (1,\infty)$ and $\theta \in (1/2p,1]$ with $\theta
\neq 1/2 + 1/2p$. For $u_{\Gamma }\in W^{2\theta -1/p,p}(\Gamma ),$ we call $%
u\in H^{2\theta ,p}(\Omega )$ a solution of \eqref{Diri-nonlinear} if 
\begin{equation}
u=\mathcal{R}_{\lambda }(F(u),u_{\Gamma }).  \label{def-solution}
\end{equation}
\end{definition}

If $\theta =1$, then such $u$ is a strong solution, meaning that it
satisfies \eqref{Diri-nonlinear} almost everywhere in $\Omega $ in the sense
of weak derivatives.

\begin{lemma}
\label{Diri-nonlinear-lemma} Let $p\in (1,\infty)$, $\theta \in (1/2p,1)$
with $\theta \neq 1/2 + 1/2p$ and assume \eqref{opA} and \eqref{f}. Then
there is $\lambda_*\geq \lambda_D$ such that for all $\lambda \geq \lambda_*$
the following holds. For all $u_\Gamma \in W^{2\theta-1/p,p}(\Gamma)$ there
is a unique solution $u = \mathcal{D}_\lambda(u_\Gamma)$ of %
\eqref{Diri-nonlinear} in the sense of \eqref{def-solution}, and the
solution operator 
\begin{equation*}
\mathcal{D}_\lambda: W^{2\theta-1/p,p}(\Gamma)\to H^{2\theta,p}(\Omega)
\end{equation*}
is globally Lipschitz continuous. If $F\in C^k(H^{2\theta,p}(\Omega),
L^p(\Omega))$ for $k\in \mathbb{N}\cup\{\infty\}$, then $\mathcal{D}_\lambda
\in C^k(W^{2\theta-1/p,p}(\Gamma), H^{2\theta,p}(\Omega))$.
\end{lemma}

\begin{proof}
For $u_\Gamma \in W^{2\theta-1/p,p}(\Gamma)$ the map $u\mapsto \mathcal{R}%
_\lambda(u,u_\Gamma)$ is a strict contraction on $H^{2\theta,p}(\Omega)$ if $%
\lambda_*$ is sufficiently large, since 
\begin{align*}
\|\mathcal{R}_\lambda(u,u_\Gamma) - \mathcal{R}_\lambda(v,u_\Gamma)\|_{H^{2%
\theta,p}(\Omega)} & = \|\mathcal{R}_\lambda (F(u)-F(v),
0)\|_{H^{2\theta,p}(\Omega)} \\
& \leq C_D \lambda_*^{-(1-\theta)} c_F \|u-v \|_{H^{2\theta,p}(\Omega)}
\end{align*}
for $u,v\in H^{2\theta,p}(\Omega)$ by Lemma \ref{Diri-lemma} and \eqref{f}.
The resulting unique fixed point is the unique solution of %
\eqref{def-solution}. The global Lipschitz continuity of the solution
operator $\mathcal{D}_\lambda$ follows from 
\begin{align*}
\|\mathcal{D}_\lambda(u_\Gamma)& - \mathcal{D}_\lambda(v_\Gamma)\|_{H^{2%
\theta,p}(\Omega)} \\
&\, = \|\mathcal{R}_\lambda(F(\mathcal{D}_\lambda(u_\Gamma)) - F(\mathcal{D}%
_\lambda(v_\Gamma)), u_\Gamma - v_\Gamma)\|_{H^{2\theta,p}(\Omega)} \\
&\, \leq C_D \lambda_*^{-(1-\theta)} c_F \|\mathcal{D}_\lambda(u_\Gamma) - 
\mathcal{D}_\lambda(v_\Gamma)\|_{H^{2\theta,p}(\Omega)} + C_\lambda
\|u_\Gamma - v_\Gamma \|_{W^{2\theta-1/p,p}(\Gamma)},
\end{align*}
and from $C_D \lambda_*^{-(1-\theta)} c_F < 1$. Finally, suppose that $F\in
C^k$ and consider the map 
\begin{equation*}
\mathcal{F }: H^{2\theta,p}(\Omega)\times W^{2\theta-1/p,p}(\Gamma) \to
H^{2\theta,p}(\Omega), \quad \mathcal{F}(u,u_\Gamma) = u - \mathcal{R}%
_\lambda(F(u), u_\Gamma).
\end{equation*}
The unique zero of $\mathcal{F}(\cdot, u_\Gamma)$ is $\mathcal{D}%
_\lambda(u_\Gamma)$. We have $\mathcal{F }\in C^k$, and for $h\in
H^{2\theta,p}(\Omega)$ it holds $D_1 \mathcal{F}(u,u_\Gamma) h = h - 
\mathcal{R}_\lambda(F^{\prime}(u) h, 0).$ As above we can estimate 
\begin{equation*}
\|\mathcal{R}_\lambda(F^{\prime}(u) h, 0)\|_{H^{2\theta,p}(\Omega)} \leq C_D
\lambda_*^{-(1-\theta)} c_F \|h\|_{H^{2\theta,p}(\Omega)},
\end{equation*}
which yields that $D_1 \mathcal{F}(u,u_\Gamma)$ is invertible for all $%
(u,u_\Gamma)$ if $\lambda_*$ is large. Thus $\mathcal{D}_\lambda \in C^k$ by
the implicit function theorem.
\end{proof}

\begin{remark}
The number $\lambda_*$ can be chosen such that $\lambda_* > \max\{\lambda_D,
(C_D c_F)^{\frac{1}{1-\theta}}\}$. It tends to infinity as the global
Lipschitz constant of $F$ tends to infinity.
\end{remark}

\subsection{Quasilinear boundary conditions of reactive-diffusive type}

\label{ldsp}

We consider the following class of elliptic problems with quasilinear,
nondegenerate dynamic boundary conditions of reactive-diffusive type, 
\begin{equation}
\left\{ 
\begin{array}{ll}
\lambda u + \mathcal{A }u=F(u) & \text{in }(0,T)\times \Omega, \\ 
\partial _{t}u_{\Gamma }+ \mathcal{C }(u_\Gamma)u_{\Gamma }+ \mathcal{B }%
(u)=G(u_{\Gamma }) & \text{on } (0,T)\times \Gamma , \\ 
u_\Gamma|_{t=0}=u_{0} & \text{on }\Gamma .%
\end{array}%
\right.  \label{ell-dyn}
\end{equation}
This is a generalization of the prototype model \eqref{ell-dyn-intro} from
the introduction.

We assume that $\mathcal{A}$ is as above and that $F$ satisfies \eqref{f}
with $\theta = 1-1/(2p)$. The nonlinear boundary differential operator $%
\mathcal{C}$ is given by 
\begin{equation}  \label{opC}
\mathcal{C }(u_\Gamma)v_\Gamma = - \text{div}_\Gamma \big(%
\delta(\cdot,u_\Gamma) \nabla_\Gamma v_\Gamma \big), \quad \delta \in
C^\infty(\Gamma\times \mathbb{R},\mathbb{R}), \quad \delta \geq \delta_* > 0,
\end{equation}
where $\delta_*$ is a constant. The nonlinear map $\mathcal{B}$ couples the
equations in the domain and on the boundary in a nontrivial way. Given $p\in
(1,\infty)$, we assume that 
\begin{equation}  \label{opB}
\mathcal{B }: H^{2-1/p,p}(\Omega)\to L^p(\Gamma) \text{ is locally
Lipschitzian}.
\end{equation}
We do not impose any further structural condition for $\mathcal{B}$. In
fact, it could vanish identically.

\begin{example}
The prototype for $\mathcal{B}$ is $\mathcal{B}(u) = B\nu \cdot (\nabla
u)|_\Gamma$ for some $B\in C^\infty(\Gamma, \mathbb{R}^{n\times n})$, which
satisfies \eqref{opB} if $p > 2$. For $B = \pm \text{id}$ one obtains $%
\mathcal{B }= \pm \partial_\nu$. In the semilinear case one can also allow $%
p\leq 2$ for such $\mathcal{B}$, see Proposition \ref{semilinear} below.
\end{example}

Next, for the boundary nonlinearity $G$ we assume 
\begin{align}
G:W^{2-2/p,p}(\Gamma )\rightarrow L^{p}(\Gamma )\text{ is locally
Lipschitzian.}  \label{g}
\end{align}

\begin{example}
If $G(u_\Gamma)(x) = g(u_\Gamma(x))$ and $g:\mathbb{R}\to \mathbb{R}$ is
locally Lipschitzian, then $G$ satisfies \eqref{g} if $p > \frac{n+1}{2}$.
If $g $ depends in addition on $\nabla_\Gamma u_\Gamma$, then $p > n+1$ is
required. These assertions are easily verified using Sobolev's embeddings.
If $g$ is polynomial, then the values of $p$ can be lowered.
\end{example}

We are now ready to precisely\ state what we mean by a \emph{strong}
solution to problem (\ref{ell-dyn}).

\begin{definition}
\label{notion_ell_dyn}A function $u$ is said to be a strong solution of %
\eqref{ell-dyn} if it is a strong solution of the elliptic equation almost
everywhere in $(0,T)$ and if its trace $u_{\Gamma }$ is a strong solution of
the parabolic equation on $(0,T)\times \Gamma $.
\end{definition}

We have the following local well-posedness result for \eqref{ell-dyn}.

\begin{theorem}
\label{thm-ell-dyn} Let $p\in (n+1,\infty )$ and assume \eqref{opA}, %
\eqref{f} with $\theta = 1-1/(2p)$, \eqref{opC}, \eqref{opB} and \eqref{g}.
Then there is $\lambda _{\ast }$ such that for all $\lambda \geq \lambda
_{\ast }$ the following holds true. The problem \eqref{ell-dyn} generates a
compact local semiflow of strong solutions on $W^{2-2/p,p}(\Gamma )$. For
all $T\in (0,t^{+}(u_{0}))$ a solution $u=u(\cdot ,u_{0})$ enjoys the
regularity 
\begin{equation*}
u\in C([0,T];H^{2-1/p,p}(\Omega ))\cap L^{p}(0,T;W^{2,p}(\Omega )),
\end{equation*}%
\begin{equation*}
u_{\Gamma }\in W^{1,p}(0,T;L^{p}(\Gamma ))\cap L^{p}(0,T;W^{2,p}(\Gamma )).
\end{equation*}%
%
%
%
%
%
%
%
%
%
%
%
%
%
%
%
\end{theorem}

\begin{remark}
\label{remark-after-thm} \ 

\begin{enumerate}
\item The corresponding result for boundary conditions of purely reactive
type, i.e., $\mathcal{C}\equiv 0$ and $\mathcal{B}=d\partial _{\nu }$, was
shown in \cite{Escher92}. There the result is based on the generation
properties of the Dirichlet-to-Neumann operator and thus requires a good
sign of the normal derivative. Moreover, the solutions enjoy worse
regularity properties up to $t=0$. In presence of the surface diffusion
operator $\mathcal{C}$, local well-posedness becomes essentially independent
of the lower order coupling $\mathcal{B}$. The latter was already observed
in \cite{VV09} for a linear problem in the special case $\mathcal{B}%
=-\partial _{\nu }$.

\item The proof shows that one can take $\lambda _{\ast }=\max \{\lambda
_{D},(C_{D}c_{F})^{p}\}.$ If $F$ is not globally Lipschitzian, then
nonexistence, nonuniqueness and noncontinuation phenomena can occur (see 
\cite{FP99}).
\end{enumerate}
\end{remark}

\begin{proof}[Proof of Theorem \protect\ref{thm-ell-dyn}]
\emph{Step 1.} Let $\mathcal{D}_{\lambda }:W^{2-2/p,p}(\Gamma)\to
H^{2-1/p,p}(\Omega)$ be the solution operator from Lemma \ref%
{Diri-nonlinear-lemma} for the nonlinear Dirichlet problem %
\eqref{Diri-nonlinear}. Given $u_0\in W^{2-2/p,p}(\Gamma)$, we consider the
quasilinear evolution equation 
\begin{equation}
\left\{ 
\begin{array}{ll}
\partial _{t}u_{\Gamma }+\mathcal{C}(u_{\Gamma })u_{\Gamma }=G(u_{\Gamma })-%
\mathcal{B}(\mathcal{D}_{\lambda }(u_{\Gamma })) & \text{on }(0,T)\times
\Gamma , \\ 
u_{\Gamma }|_{t=0}=u_{0} & \text{on }\Gamma .%
\end{array}%
\right.  \label{quasi-KPW}
\end{equation}%
for $u_\Gamma$. We verify the conditions of \cite[Theorems 2.1 and 3.1,
Corollary 3.2]{KPW10} to apply the abstract results on local well-posedness
provided there.

By assumption and the Lipschitz continuity of $\mathcal{D}_\lambda$, the map 
$u_\Gamma \mapsto G(u_\Gamma) - \mathcal{B}(\mathcal{D}_\lambda(u_\Gamma))$
is locally Lipschitzian $W^{2-2/p,p}(\Gamma)\to L^p(\Gamma)$. Next we
rewrite the leading term $\mathcal{C}(u_\Gamma) u_\Gamma$ into 
\begin{equation*}
\mathcal{C}(u_\Gamma) u_\Gamma = - \delta(\cdot, u_\Gamma) \Delta_\Gamma
u_\Gamma - \nabla_\Gamma (\delta(\cdot, u_\Gamma)) \nabla_\Gamma u.
\end{equation*}
For the first term we have 
\begin{equation*}
\|\delta(\cdot, u_\Gamma)\Delta_\Gamma w_\Gamma -
\delta(\cdot,v_\Gamma)\Delta_\Gamma w_\Gamma\|_{L^p(\Gamma)} \leq
\|\delta(\cdot, u_\Gamma) - \delta(\cdot, v_\Gamma)\|_{L^\infty(\Gamma)}
\|w_\Gamma\|_{W^{2,p}(\Gamma)}.
\end{equation*}
The superposition operator induced by $\delta$ is locally Lipschitzian as a
map $C(\Gamma) \to C(\Gamma)$, and since $p > \frac{n+1}{2}$ we have that $%
W^{2-2/p,p}(\Gamma) \hookrightarrow C(\Gamma)$. Therefore $u_\Gamma \mapsto
- \delta(\cdot, u_\Gamma) \Delta_\Gamma $ is locally Lipschitzian as a map $%
W^{2-2/p,p}(\Gamma) \to \mathcal{L}(W^{2,p}(\Gamma), L^p(\Gamma))$. Further,
for each $u_\Gamma \in W^{2-2/p,p}(\Gamma)$ the function $%
-\delta(\cdot,u_\Gamma)$ belongs to $C(\Gamma)$. Hence by Theorem \ref{thm1}%
, the operator $-\delta(\cdot, u_\Gamma) \Delta_\Gamma$ with domain $%
W^{2,p}(\Gamma)$ on $L^p(\Gamma)$ enjoys the property of maximal $L^p$%
-regularity. Finally, if $p> n+1$ then $W^{2-2/p,p}(\Gamma) \hookrightarrow
C^1(\Gamma)$, and this implies that $u_\Gamma \mapsto \nabla_\Gamma
(\delta(\cdot, u_\Gamma)) \nabla_\Gamma u$ is locally Lipschitzian $%
W^{2-2/p,p}(\Gamma) \to L^p(\Gamma)$.

It thus follows from the results of \cite{KPW10} that \eqref{quasi-KPW}
generates a local solution semiflow on $W^{2-2/p,p}(\Gamma )$, such that 
\begin{equation*}
u_{\Gamma }\in W^{1,p}(0,T;L^{p}(\Gamma ))\cap L^{p}(0,T;W^{2,p}(\Gamma
))\cap C([0,T];W^{2-2/p,p}(\Gamma ))
\end{equation*}%
for each $T\in (0,t^{+}(u_{0}))$. Hence $u:=\mathcal{D}_{\lambda }(u_{\Gamma
})$ solves \eqref{ell-dyn}, as described in Definition \ref{notion_ell_dyn},
by the Lemmas \ref{Diri-lemma} and \ref{Diri-nonlinear-lemma}. In this way
the semiflow for \eqref{quasi-KPW} becomes a semiflow for \eqref{ell-dyn}.

\emph{Step 2.} It remains to show the compactness of the semiflow generated
by \eqref{ell-dyn}. To this end we modify the arguments of \cite[Section 3]%
{KPW10} appropriately. We will use the notion and properties of the weighted
spaces $L_{\mu }^{p}$ used in \cite{KPW10}, which are given by 
\begin{equation*}
L_{\mu }^{p}(0,T;E):=\bigg\{ v:(0,T)\rightarrow E\,:\,\text{ }\Vert v\Vert
_{L_{\mu }^{p}(0,T;E)}^{p}:=\int_{0}^{T}t^{p(1-\mu
)}|v(t)|_{E}^{p}\,dt<+\infty \bigg\} ,
\end{equation*}%
for some $p\in (1,\infty )$, $\mu \in (1/p,1],$ and a Banach space $E$ with
norm $|\cdot |_{E}$. The corresponding Sobolev spaces $W_{\mu }^{1,p}(0,T;E)$
are defined by 
\begin{equation*}
W_{\mu }^{1,p}(0,T;E):=\{v\in L_{\mu }^{p}(0,T;E)\,:\, \exists\,v^{\prime
}\in L_{\mu }^{p}(0,T;E)\}.
\end{equation*}

Let us now return to the proof. By assumption there exists a number $\mu \in
(1/p,1)$ with $2\mu -2/p>1+\frac{n-1}{p}$. The same arguments as above show
that $u_{\Gamma }\mapsto -\delta (\cdot ,u_{\Gamma })\Delta _{\Gamma }$ is
locally Lipschitzian $W^{2\mu -2/p,p}(\Gamma )\rightarrow \mathcal{L}%
(W^{2,p}(\Gamma ),L^{p}(\Gamma ))$, and the lower order nonlinearities are
locally Lipschitzian $W^{2\mu -2/p,p}(\Gamma )\rightarrow L^{p}(\Gamma )$.
Thus by \cite[Theorem 2.1]{KPW10}, for each $v_{0}\in W^{2\mu -2/p,p}(\Gamma
)$ there are $r,T>0$ and a continuous map 
\begin{equation*}
\Phi :B_{r}(v_{0})\subset W^{2\mu -2/p,p}(\Gamma )\rightarrow W_{\mu
}^{1,p}(0,T;L^{p}(\Gamma ))\cap L_{\mu }^{p}(0,T;W^{2,p}(\Gamma ))
\end{equation*}%
such that $u_{\Gamma }=\Phi (u_{0})$ solves \eqref{quasi-KPW} on $(0,T)$.

Now let $M$ be a bounded subset of $W^{2-2/p,p}(\Gamma)$ such that $%
t^+(M)\geq T > 0$. Then $M$ is relatively compact in $W^{2\mu-2/p,p}(\Gamma)$%
. Hence finitely many balls $B_{r_i} \subset W^{2\mu-2/p,p}(\Gamma)$ suffice
to cover $M$, with corresponding solution maps $\Phi_i$ and times $T_i$ as
above. Let $T_0 = \min T_i$, and take $0 < t\leq T_0$. Then we have $u(t;M)
= \bigcup_{i} \text{tr}_{t} \Phi_i(B_i \cap M)$. By the continuity of $%
\Phi_i $, the set $\Phi_i(B_i \cap M)$ is relatively compact in $%
W^{1,p}_\mu(0,T_0;L^p(\Gamma)) \cap L^p_\mu(0,T_0; W^{2,p}(\Gamma))$.
Moreover, the trace $\text{tr}_{t}$ at time $t$ is continuous from the
latter space into the higher regularity space $W^{2-2/p,p}(\Gamma)$, due to
the fact that the weight $t^{p(1-\mu)}$ only has an effect at $t=0$ (see 
\cite[Proposition 3.1]{PS04} and \cite[Theorem 4.2]{MS11a}). Thus $u(t;M)$
is relatively compact in $W^{2-2/p,p}(\Gamma)$. Finally, in case $T_0 < t
\leq T $ we obtain the relative compactness of $u(t;M)$ from $u(t;M) =
u(t-T_0; u(T_0;M))$ and the continuity of $u(t-T_0; \cdot)$ on $%
W^{2-2/p,p}(\Gamma)$.
\end{proof}

Things are simpler in the semilinear case.

\begin{proposition}
\label{semilinear} Let $p\in (1,\infty)$ and assume \eqref{opA} and %
\eqref{opC}, where $\delta $ is independent of $u_\Gamma$. Suppose that
there is $\theta\in (1/2p,1)$ such that 
\begin{equation*}
F: H^{2\theta,p}(\Omega) \to L^p(\Omega), \quad G: W^{2\theta-1/p,p}(\Gamma)
\to L^p(\Gamma), \quad \mathcal{B }: H^{2\theta,p}(\Omega) \to L^p(\Gamma),
\end{equation*}
where $F$ is globally Lipschitzian and $G,\mathcal{B}$ are Lipschitzian on
bounded sets. Then there is $\lambda_*$ such that for all $\lambda \geq
\lambda_*$ the following holds. For all $\sigma \in (\theta, 1)$ the problem %
\eqref{ell-dyn} generates a compact local semiflow of strong solutions on $%
W^{2\sigma-1/p,p}(\Gamma)$. A solution $u = u(\cdot;u_0)$ enjoys the
regularity 
\begin{equation*}
u\in C([0,t^+); H^{2 \sigma,p}(\Omega)) \cap C(0,t^+; W^{2,p}(\Omega)),
\end{equation*}
\begin{equation*}
u_\Gamma \in C([0,t^+); W^{2 \sigma- 1/p,p}(\Gamma)) \cap C^1(0,t^+;
L^p(\Gamma)) \cap C(0,t^+; W^{2,p}(\Gamma)).
\end{equation*}
\end{proposition}

\begin{proof}
The reformulation \eqref{quasi-KPW} of \eqref{ell-dyn} is now an abstract
semilinear problem. Since $W^{2,p}(\Gamma) \hookrightarrow L^p(\Gamma)$ is
compact, the assertions follow from Theorem \ref{thm1}, the Lemmas \ref%
{Diri-lemma} and \ref{Diri-nonlinear-lemma}, and e.g. \cite[Theorems 2.1.1
and 3.2.1, Corollary 2.3.1]{CD}.
\end{proof}

\subsection{Compactness in the purely reactive case}

\label{cprc}

We complement the results in \cite{Escher92, Escher94} concerning
compactness of the solution semiflow. We consider problems of type 
\begin{equation}
\left\{ 
\begin{array}{ll}
\lambda u - \text{div}(d \nabla u)=f(u) & \text{in } (0,T)\times \Omega , \\ 
\partial _{t}u_{\Gamma }+ d\partial_\nu u = g(u_\Gamma) & \text{on }
(0,T)\times \Gamma , \\ 
u_\Gamma|_{t=0}=u_{0} & \text{on }\Gamma.%
\end{array}%
\right.  \label{ell-dyn-classic-escher}
\end{equation}
Throughout this subsection we assume that 
\begin{equation}  \label{reactive-assu}
d\in C^\infty(\overline{\Omega}), \qquad d\geq d_* > 0, \qquad f,g \in
C^\infty(\mathbb{R}), \qquad |f^{\prime}|\leq c_f.
\end{equation}
The results of \cite[Theorem 6.2]{Escher92} and \cite[Theorem 2]{Escher94}
can be summarized as follows.

\begin{proposition}
\label{escher-results} Assume \eqref{reactive-assu}. Then for $p \in
(n,\infty)$ there is $\lambda_*$ such that for all $\lambda\geq \lambda_*$
the problem \eqref{ell-dyn-classic-escher} generates a local semiflow of
classical solutions on $W^{1-1/p,p}(\Gamma)$. A solution $u$ enjoys the
regularity 
\begin{equation*}
u \in C([0,t^+); W^{1,p}(\Omega)) \cap C^1(0,t^+; C^\infty(\Gamma)) \cap
C(0,t^+; C^\infty(\overline{\Omega})).
\end{equation*}
\end{proposition}

For sufficiently large $\lambda$ we define the Dirichlet-Neumann operator $%
\mathcal{N}_\lambda $ by 
\begin{equation*}
\mathcal{N}_\lambda u_\Gamma := d\partial_\nu \mathcal{R}_\lambda(0,u_%
\Gamma),
\end{equation*}
where $\mathcal{R}_\lambda$ is from Lemma \ref{Diri-lemma}. The following
generator properties of $\mathcal{N}_\lambda$ are shown in \cite[Theorem 1.5]%
{Escher92} (see also \cite[Theorem 3]{Escher94} and \cite[Section 6]{ES08}).

\begin{proposition}
\label{D-N-lemma} Let $d\in C^\infty(\overline{\Omega})$ with $d\geq d_* > 0$%
, and let $\lambda$ be sufficiently large. Then for all $p\in (1,\infty)$
and $\theta \geq 0$ the operator $-\mathcal{N}_\lambda$ with domain $%
W^{\theta+1,p}(\Gamma)$ generates an analytic $C_0$-semigroup on $%
W^{\theta,p}(\Gamma)$.
\end{proposition}


Now we prove the compactness of the semiflow generated by %
\eqref{ell-dyn-classic-escher}.

\begin{proposition}
\label{compact-reactive} For each $p\in (n,\infty)$, the local solution
semiflow from Proposition \ref{escher-results} is compact.
\end{proposition}

\begin{proof}
Using the solution operator $\mathcal{D}_\lambda$ from Lemma \ref%
{Diri-nonlinear-lemma} for the nonlinear Dirichlet problem %
\eqref{Diri-nonlinear}, the regularity of the solutions allows rewrite %
\eqref{ell-dyn-classic-escher} into 
\begin{equation}
\left\{ 
\begin{array}{ll}
\label{abs-semi-N} \partial_t u_\Gamma + \mathcal{N}_\lambda u_\Gamma =
g(u_\Gamma) - d\partial_\nu \mathcal{R}_\lambda(f(\mathcal{D}%
_\lambda(u_\Gamma)),0) & \text{on } (0,T)\times \Gamma , \\ 
u_\Gamma |_{t=0}=u_{0} & \text{on }\Gamma,%
\end{array}%
\right.
\end{equation}
which is a semilinear problem for $u_\Gamma$. Let $M\subset
W^{1-1/p,p}(\Gamma)$ be bounded with $t^+(M)\geq T > 0$. Fix $t\in (0,T)$.
Note that $D(\mathcal{N}_\lambda^{\alpha_2}) \hookrightarrow W^{s,p}(\Gamma)
\hookrightarrow D(\mathcal{N}_\lambda^{\alpha_1})$ for $\alpha_2 > s>
\alpha_1\geq 0$. If $\alpha$ is sufficiently close to $1-1/p$, then
Sobolev's embedding implies that $u_\Gamma \mapsto g(u_\Gamma)$ is
Lipschitzian on bounded sets as a map $D(\mathcal{N}_\lambda^\alpha) \to
L^p(\Gamma)$. Using the Lemmas \ref{Diri-lemma} and \ref%
{Diri-nonlinear-lemma} and the Lipschitz properties of $f$ and $\mathcal{D}%
_\lambda$, for $u_\Gamma, v_\Gamma \in D(\mathcal{N}_\lambda^\alpha)$ and $%
\eta \in (0,\alpha)$ we estimate 
\begin{align*}
\|d\partial_\nu \mathcal{R}_\lambda(f(\mathcal{D}_\lambda(u_\Gamma)) - f(%
\mathcal{D}_\lambda(v_\Gamma)),0)\|_{L^p(\Gamma)}&\, \leq C\|f(\mathcal{D}%
_\lambda(u_\Gamma)) - f(\mathcal{D}_\lambda(v_\Gamma))\|_{L^p(\Omega)} \\
&\, \leq C \|u_\Gamma - v_\Gamma \|_{W^{\eta,p}(\Gamma)} \\
&\, \leq C \|u_\Gamma - v_\Gamma \|_{D(\mathcal{N}_\lambda^\alpha)}.
\end{align*}
Hence $u_\Gamma \mapsto d\partial_\nu \mathcal{R}_\lambda(f(\mathcal{D}%
_\lambda(u_\Gamma)),0)$ is globally Lipschitzian as a map $D(\mathcal{N}%
_\lambda^\alpha)\to L^p(\Gamma)$. Therefore \cite[Proposition 3.2.1]{CD}
applies to \eqref{abs-semi-N}, and we obtain that $u_\Gamma(t;M)$ is bounded
in $D(\mathcal{N}_\lambda^\alpha)$ for all $\alpha\in (0,1)$. Since $%
W^{1,p}(\Gamma) \hookrightarrow L^p(\Gamma)$ is compact, we conclude that $%
u_\Gamma(t;M)$ is relatively compact in $W^{1-1/p,p}(\Gamma)$.
\end{proof}

\section{Qualitative properties of classical solutions}

\label{qp}

In this section we study the qualitative properties of solutions of the
semilinear problem 
\begin{equation}
\left\{ 
\begin{array}{ll}
\lambda u-\text{div}(d\nabla u)=f(u) & \text{in }(0,T)\times \Omega , \\ 
\partial _{t}u_{\Gamma }-\text{div}_{\Gamma }(\delta \nabla _{\Gamma
}u_{\Gamma })+d\partial _{\nu }u=g(u_{\Gamma }) & \text{on }(0,T)\times
\Gamma , \\ 
u_\Gamma|_{t=0}=u_{0} & \text{on }\Gamma ,%
\end{array}%
\right.  \label{ell-dyn-classic}
\end{equation}%
where we assume throughout that $\lambda \geq \lambda _{\ast }$ is
sufficiently large (in dependence on the other parameters). We treat the two
types of boundary conditions simultaneously and assume that 
\begin{equation}
\left\{ 
\begin{array}{c}
d\in C^{\infty }(\overline{\Omega }),\quad d\geq d_{\ast }>0,\quad f,g\in
C^{\infty }(\mathbb{R}),\quad |f^{\prime }|\leq c_{f},\quad p\in (n,\infty),
\\ 
\delta \in C^{\infty }(\Gamma ),\quad \text{and either }\delta \geq \delta
_{\ast }>0\text{ or }\delta \equiv 0.%
\end{array}%
\right.  \label{classical-assum}
\end{equation}%
The local well-posedness of \eqref{ell-dyn-classic} is provided by the
Propositions \ref{semilinear} and \ref{escher-results}. To simplify the
notation we set 
\begin{equation*}
\mathcal{X}_{\delta }:=W^{2-2/p,p}(\Gamma )\quad \text{if }\delta \geq
\delta _{\ast },\qquad \mathcal{X}_{\delta }:=W^{1-1/p,p}(\Gamma )\quad 
\text{if }\delta \equiv 0,
\end{equation*}%
for the corresponding phase spaces. We will make essential use of the fact
that by the Propositions \ref{semilinear} and \ref{escher-results}, for both
types of boundary conditions the trace $u_\Gamma$ of a strong resp.
classical solution of \eqref{ell-dyn-classic} satisfies 
\begin{equation}
\left\{ 
\begin{array}{ll}
\label{abs-semi} \partial_t u_\Gamma + \mathcal{C }u_\Gamma + \mathcal{N}%
_\lambda u_\Gamma = g(u_\Gamma) - d\partial_\nu \mathcal{R}_\lambda(f(%
\mathcal{D}_\lambda(u_\Gamma)),0) & \text{on } (0,T)\times \Gamma , \\ 
u_\Gamma |_{t=0}=u_{0} & \text{on }\Gamma,%
\end{array}%
\right.
\end{equation}
where $\mathcal{C }u_\Gamma = - \text{div}_\Gamma (\delta \nabla_\Gamma
u_\Gamma)$, $\mathcal{N}_\lambda$ is the Dirichlet-Neumann operator, $%
\mathcal{R}_\lambda$ is from Lemma \ref{Diri-lemma} and $\mathcal{D}_\lambda$
is from Lemma \ref{Diri-nonlinear-lemma}.

By Theorem \ref{thm1} and Proposition \ref{D-N-lemma}, the operators $%
\mathcal{C}$ and $\mathcal{N}_\lambda$ are both the negative generators of
an analytic $C_0$-semigroup on $L^p(\Gamma)$. Therefore we may represent $%
u_\Gamma$ by the variation of constants formula with an inhomogeneity as
above.

\subsection{Classical solutions}

\label{classical-solutions}

We show the smoothness of solutions in space and time. Besides its own
interest, this will become important to apply the comparison result Lemma %
\ref{comparison} below and to show that \eqref{ell-dyn-classic} is of
gradient structure (see Section \ref{attractors}).

The key to smoothness in time is the following.

\begin{lemma}
\label{time-reg-lem} Suppose that \eqref{classical-assum} holds, and that $%
\varphi \in C^\infty (0,T; W^{1-1/p,p}(\Gamma))$. For each $t\in (0,T)$,
denote by $u= u(t,\cdot)$ the unique solution of 
\begin{equation}
\left\{ 
\begin{array}{ll}
\lambda u - \emph{\text{div}}(d \nabla u)= f(u) & \emph{\text{in }}\Omega ,
\\ 
u|_\Gamma = \varphi(t) & \emph{\text{on }}\Gamma,%
\end{array}%
\right.  \label{Diri-sm}
\end{equation}
i.e., $u = \mathcal{R}_\lambda(f(u), \varphi(t))$ for $t\in (0,T)$. Then $%
u\in C^\infty (0,T; H^{1,p}(\Omega))$.
\end{lemma}

\begin{proof}
Define $\mathcal{F}:(0,T)\times H^{1,p}(\Omega )\rightarrow H^{1,p}(\Omega )$
by 
\begin{equation*}
\mathcal{F}(t,v):=v-\mathcal{R}_{\lambda }\big (f(v),\varphi (t)\big).
\end{equation*}%
By Lemma \ref{Diri-nonlinear-lemma}, for each $t$ the unique zero of $%
\mathcal{F}$ is $u(t,\cdot )$. The assumption on $p$ guarantees that the
superposition operator $v\mapsto f(v)$ belongs to $C^{\infty
}(H^{1,p}(\Omega ),L^{p}(\Omega ))$, with derivative $h\mapsto f^{\prime
}(v)h$. The regularity of $\varphi $ and the continuity of $\mathcal{R}%
_{\lambda }$ thus show that $\mathcal{F}\in C^{\infty }$. At $v\in
H^{1,p}(\Omega )$ the derivative $D_{2}\mathcal{F}(t,v)$ is given by $%
h\mapsto h-\mathcal{R}_{\lambda }(f^{\prime }(v)h,0),$ and by Lemma \ref%
{Diri-lemma} it holds 
\begin{equation*}
\Vert \mathcal{R}_{\lambda }(f^{\prime }(v)h,0)\Vert _{H^{1,p}(\Omega )}\leq
C_{D}\lambda _{\ast }^{-1/2}c_{f}\Vert h\Vert _{H^{1,p}(\Omega )}.
\end{equation*}%
Therefore $D_{2}\mathcal{F}(t,v)$ is invertible for all $t$ and all $v$. We
obtain that for every $t_{0}\in (0,T)$ there are $\varepsilon >0$ and a
function $\Phi \in C^{\infty }\big (t_{0}-\varepsilon ,t_{0}+\varepsilon
;H^{1,p}(\Omega )\big)$ such that $\mathcal{F}(t,\Phi(t)) = 0$. Uniqueness
implies that $\Phi (t)=u(t,\cdot )$ for all $t\in (t_{0}-\varepsilon
,t_{0}+\varepsilon )$. Hence $u\in C^{\infty }\big(0,T;H^{1,p}(\Omega )\big)$
as asserted.
\end{proof}

After this preparation we can show the smoothness of solutions.

\begin{proposition}
\label{classic-diff} Let \eqref{classical-assum} hold. Then for all $u_0\in 
\mathcal{X}_\delta$ the solution $u$ of \eqref{ell-dyn-classic} satisfies $%
u\in C^\infty ((0,t^+) \times \overline{\Omega}).$
\end{proposition}

\begin{proof}
Throughout we fix $T < t^+$.

\emph{Step 1.} First let $\delta \geq \delta_*$. For sufficiently large $%
\rho $ the operator $\rho + \mathcal{C}$ is invertible and commutes with $-%
\mathcal{C}$. Employing local arguments as in the proof of Theorem \ref{thm1}
and interpolation, we obtain that $\rho +\mathcal{C}$ is an isomorphism $%
W^{2+\theta,p}(\Gamma)\to W^{\theta,p}(\Gamma)$ for all $\theta \geq 0$.
Thus $-\mathcal{C}$ with domain $W^{2+\theta,p}(\Gamma)$ generates an
analytic $C_0$-semigroup on $W^{\theta,p}(\Gamma)$ for all $\theta$.

The trace $u_\Gamma$ may be represented by 
\begin{equation}  \label{mild-sol}
u_\Gamma(t,\cdot) = e^{- \mathcal{C }t }u_0 + e^{- \mathcal{C}\cdot } *\big (%
g(u_\Gamma) - d \partial_\nu u\big )(t), \quad t\in (0,T].
\end{equation}
Since $u \in C([0,T]; H^{2-1/p}(\Omega))$ we have $g(u_\Gamma) - d
\partial_\nu u\in C([0,T]; W^{1-2/p,p}(\Gamma))$. We may thus consider %
\eqref{mild-sol} as an identity on $W^{1-2/p,p}(\Gamma)$, and obtain from 
\cite[Corollary 4.3.9]{Lun95} that 
\begin{equation*}
u_\Gamma \in C^1 ((0,T]; W^{1-2/p,p}(\Gamma)) \cap C ((0,T];
W^{3-2/p,p}(\Gamma)).
\end{equation*}
Since $u = \mathcal{R}_\lambda (f(u), u_\Gamma)$ and $f(u(t,\cdot)) \in
H^{1-1/p}(\Omega)$, we further obtain from \cite[Theorem 13.1]{Amann84} that 
$u(t,\cdot)\in H^{3-1/p,p}(\Omega)$ for all $t$, and that 
\begin{align}
\|u(t_1,\cdot) - u(t_2,\cdot)\|_{H^{3-1/p,p}(\Omega)} &\, \leq C\big( %
\|f(u(t_1,\cdot)) - f(u(t_2,\cdot))\|_{H^{1-1/p}(\Omega)}  \notag \\
&\, \qquad\quad + \|u_\Gamma(t_1,\cdot)- u_\Gamma(t_2,\cdot)\|_{
W^{3-2/p,p}(\Gamma)}\big)  \label{class-est}
\end{align}
with a constant $C$ independent of $t_1,t_2\in (0,T]$. Thus $u \in C((0,T];
W^{3-1/p,p}(\Omega))$. An iteration of these arguments together with
Sobolev's embeddings gives 
\begin{equation*}
u \in C^1((0,T]; C^\infty(\Gamma))\cap C((0,T]; C^\infty(\overline{\Omega})).
\end{equation*}
Now it follows from \eqref{mild-sol} that $u_\Gamma \in C^\infty((0,T];
C^\infty(\Gamma))$. Moreover, Lemma \ref{time-reg-lem} implies that $u\in
C^\infty((0,T]; H^{1,p}(\Omega))$.

\emph{Step 2.} Let now $\delta \equiv 0$. By Proposition \ref{escher-results}
we have $u\in C^1((0,T]; C^\infty(\Gamma)) \cap C((0,T]; C^\infty(\overline{%
\Omega}))$, and further 
\begin{equation*}
u_\Gamma(t,\cdot) = e^{-\mathcal{N}_\lambda t} u_0 + e^{-\mathcal{N}%
_\lambda\cdot } * \big (g(u_\Gamma) - d \partial_\nu \mathcal{R}%
_\lambda(f(u),0)\big)(t), \quad t\in (0,T].
\end{equation*}
As above, this formula yields $u_\Gamma\in C^\infty((0,T]\times \Gamma)$ and
then $u\in C^\infty((0,T]; H^{1,p}(\Omega))$ by Lemma \ref{time-reg-lem}.

\emph{Step 3.} For both types of boundary conditions it now follows from the
linearity and the continuity of $\mathcal{R}_\lambda$ that 
\begin{equation*}
\partial_t^k u = \mathcal{R}_\lambda (\partial_t^k (f(u)), \partial_t^k
u_\Gamma)
\end{equation*}
for all $k\in \mathbb{N}$. Now argue by induction and suppose that $%
\partial_t^{k-1} u \in C((0,T]; C^\infty(\overline{\Omega}))$. Note that $%
\partial_t^k(f(u))$ is of the form $f^{\prime}(u) \partial_t^k u + \psi$,
where $\psi \in C((0,T]; C^\infty(\overline{\Omega}))$ is a polynomial in
the derivatives of $u$ up to the order $k-1$ and derivatives of $f$ with $u$
inserted. Since $|f^{\prime}(u)| \leq c_f$ we may apply \cite[Theorem 13.1]%
{Amann84} to $\mathcal{A }- f^{\prime}(u)$ for all $\lambda \geq \lambda_D +
c_f$ and estimate as in \eqref{class-est} to obtain $\partial_t^k u \in
C((0,T]; C^\infty(\overline{\Omega}))$.
\end{proof}

\subsection{Blow-up}

\label{bl}

In this subsection we assume that $d\equiv d_* >0$ and $\delta \equiv
\delta_*\geq 0$ are constants. Our blow-up results are based on the method
of subsolutions and the following comparison lemma. Its proof is inspired by 
\cite[Theorem II.3]{Rothe}.

\begin{lemma}
\label{comparison} Assume $f,g\in C^1(\mathbb{R})$ with $|f^{\prime}|\leq
c_f $, $\lambda \geq c_f$, $d > 0$ and $\delta\geq 0$. If 
\begin{equation*}
u,v\in C([0,T]\times \overline{\Omega}) \cap C^1((0,T]; C(\Gamma)) \cap
C((0,T]; C^2(\overline{\Omega}))
\end{equation*}
satisfy 
\begin{equation*}
\left\{ 
\begin{array}{ll}
\lambda v-d\Delta v-f(v) \geq \lambda u - d \Delta u - f(u) & \emph{\text{in 
}}(0,T] \times \Omega, \\ 
\partial _{t}v_{\Gamma }-\delta \Delta _{\Gamma }v_{\Gamma }+d\partial _{\nu
}v- g(v_{\Gamma }) \geq \partial _{t}u_{\Gamma }-\delta \Delta _{\Gamma
}u_{\Gamma }+d\partial _{\nu }u- g(u_{\Gamma }) & \emph{\text{on }}%
(0,T]\times \Gamma, \\ 
v_\Gamma|_{t=0}\geq u_\Gamma|_{t=0} & \emph{\text{on }}\Gamma,%
\end{array}%
\right.
\end{equation*}%
then $v \geq u$ on $[0,T]\times \overline{\Omega}$.
\end{lemma}

\begin{proof}
The assumptions on $f$ and $\lambda$ imply that the function 
\begin{equation*}
a(t,x) = \frac{\lambda v(t,x) - f(v(t,x)) - (\lambda u(t,x) -f(u(t,x)))}{%
v(t,x) - u(t,x)}
\end{equation*}
is continuous and nonnegative on $[0,T]\times \overline{\Omega}$. Moreover,
for all $(t,x)\in [0,T]\times \overline{\Omega}$ we can write 
\begin{equation*}
g(v(t,x)) - g(u(t,x)) = (L- b(t,x))(v(t,x) -u(t,x)),
\end{equation*}
where $L>0$ is a constant and $b$ is continuous and nonnegative on $%
[0,T]\times \Gamma$. Define 
\begin{equation*}
w(t,x):= e^{Lt}(v(t,x) -u(t,x)).
\end{equation*}
We suppose that $m:= \min_{ [0,T]\times \overline{\Omega}} w < 0 $ and
derive a contradiction. Let $(t_0,x_0)\in (0,T]\times \overline{\Omega}$ be
such that $m = w(t_0,x_0)$. The function $w$ satisfies 
\begin{equation*}
\lambda w - d\Delta w \geq e^{Lt_0}(f(v)-f(u)) \qquad \text{in }%
\{t_0\}\times \Omega,
\end{equation*}
and is thus a classical solution of 
\begin{equation*}
d\Delta w - e^{Lt_0} (\lambda v -f(v) - (\lambda u - f(u))) = d \Delta w - a
w \leq 0 \qquad \text{in }\{t_0\}\times \Omega.
\end{equation*}
Since $- a \leq 0$ we deduce from the strong maximum principle \cite[Theorem
3.5]{GT} that $x_0\in \Gamma$. Now the Hopf lemma \cite[Lemma 3.4]{GT}
implies $\partial_\nu w(t_0,x_0) < 0$. Therefore 
\begin{equation}  \label{1001}
\partial_t w(t_0,x_0) - \delta \Delta_\Gamma w(t_0,x_0) + b(t_0,x_0)
w(t_0,x_0) > 0.
\end{equation}
As $b \geq 0$ we have $b(t_0,x_0) w(t_0,x_0) \leq 0$, and further $%
\partial_t w(t_0,x_0) \leq 0$ since $t\mapsto w(t,x_0)$ attains its minimum
in $t_0$. Moreover, in case $\delta > 0$, take orthogonal coordinates $\text{%
g}:U\subset \mathbb{R}^{n-1} \to \Gamma$ for $x_0\in \Gamma$, with $\text{g}%
(y_0) = x_0$ for some $y_0 \in U$. Then $y \mapsto w(t_0,\text{g}(y))$ has a
local minimum in $y_0$, which implies that $\nabla_y w(t_0,\text{g}(y_0)) =
0 $ and $\Delta_y w(t_0,\text{g}(y_0)) \geq 0$. Hence the formula for $%
\Delta_\Gamma$ in coordinates yields 
\begin{equation*}
\Delta_\Gamma w(t_0,x_0) = \Delta_\Gamma w(t_0, \text{g}(y_0)) = \Delta_y
w(t_0,\text{g}(y_0)) \geq 0.
\end{equation*}
The signs of the terms on the left-hand side of \eqref{1001} lead to a
contradiction.
\end{proof}

To obtain appropriate subsolutions we modify the ones from \cite[Lemma 4.1]%
{AMTR}.

\begin{proposition}
\label{blow-up-prop} Let \eqref{classical-assum} hold and assume $d \equiv
d_*$ and $\delta \equiv \delta_*$. Let further 
\begin{equation*}
\frac{g(\xi)}{\xi^q} \to + \infty \qquad \text{as }\xi \to +\infty,
\end{equation*}
for some $q>1$. Then there is $C>0$ such that if $u_0\in \mathcal{X}_\delta$
satisfies $u_0 \geq C$, then the solution of \eqref{ell-dyn-classic} blows
up in finite time.
\end{proposition}

\begin{remark}
For $\delta =0$, blow-up results for \eqref{ell-dyn-classic} with $f \neq 0$
were obtained in \cite{Vulkov} by the so-called concavity method.
\end{remark}

\begin{proof}
For $1<r\leq q$, let $\varphi(s) := ( c- (r-1)s)^{-1/(r-1)}$, where $c:=
(r-1) (\max_{y\in \overline{\Omega}} \sum_i y_i + 1)$, such that $%
\varphi^{\prime}= \varphi^r$ and $\varphi^{\prime\prime}= r \varphi^{2r-1}$.
Define $\underline{u}$ by 
\begin{equation*}
\underline{u}(t,x) : = \varphi\Big(\sum_i x_i + t\Big) = \Big ((r-1) \Big[ %
\max_{y\in \overline{\Omega}} \sum_i y_i - \sum_i x_i + 1 - t\Big] \Big)^{-
1/(r-1)},
\end{equation*}
which is well-defined on $\overline{\Omega}$ as long as $t < 1.$ Observe
that $\underline{u}$ is positive and that 
\begin{equation}  \label{1002}
\text{for all $K>0$ there is $r>1$ such that} \qquad \underline{u}(t,x) \geq
K \quad \text{on } [0,1)\times \overline{\Omega}.
\end{equation}
We check that $\underline{u}$ is a subsolution of \eqref{ell-dyn-classic} on 
$(0,1)\times \overline{\Omega}$ for a suitable choice of $r$.

First consider the elliptic equation. The assumption on $\lambda$ and $f$
yields 
\begin{equation*}
\lambda \underline{u} - f(\underline{u}) \leq (\lambda+c_f) \underline{u}
-f(0),
\end{equation*}
and we have $\Delta \underline{u} = n r \underline{u}^{2r-1}$. By %
\eqref{1002} (with $\underline{u}^{2(r-1)}$ instead of $\underline{u}$) we
can achieve the inequality 
\begin{equation*}
(\lambda+c_f) \leq d n r \underline{u}^{2(r-1)} + f(0)/\underline{u} \qquad 
\text{on } (0,1)\times \Omega
\end{equation*}
if $r$ is sufficiently close to $1$.

For the boundary equation we have $\partial_t \underline{u} = \underline{u}%
^r $ and $\partial_\nu \underline{u} = (\nu \cdot \mathbf{1}) \underline{u}%
^r $, where $\mathbf{1} = (1,...,1)\in \mathbb{R}^n$. To treat the
Laplace-Beltrami term in case $\delta >0$, fix $x_0\in \Gamma$ and take
orthogonal coordinates $\text{g}:U\subset \mathbb{R}^{n-1}\to \Gamma$ for $%
x_0$, such that $x_0 = \text{g}(y_0)$ for some $y_0\in U$. Let $|\text{G}|$
be the Gramian and let $\text{G}^{-1} = (\text{g}^{ij})_{i,j=1,...,n-1}$ be
the inverse fundamental form with respect to $\text{g}$. We write $a(x) =
\sum_i x_i$ for simplicity. Since $(\text{g}^{ij})_{y=y_0}$ equals the
Kronecker symbol, we have 
\begin{align*}
(\Delta_\Gamma \underline{u})(t,x_0) &\, = \sum_{i,j=1}^{n-1} \partial_i %
\big [ \sqrt{|\text{G}|} \text{g}^{ij} \partial_j(\varphi( a\circ \text{g}%
(y_0) + t))\big] \\
&\, = \underline{u}^r(t,x_0) \Delta_\Gamma a(x_0) + r \underline{u}%
^{2r-1}\sum_{i=1}^{n-1} |\partial_i (a\circ \text{g})(y_0)|^2 \\
&\, \geq m \underline{u}^r(t,x_0),
\end{align*}
where $m = \min_{x\in \Gamma} \Delta_\Gamma a(x)$. Therefore on $(0,1)\times
\Gamma$ we have 
\begin{equation*}
\partial_t \underline{u} - \delta \Delta_\Gamma \underline{u} + d
\partial_\nu \underline{u} \leq (1 - \delta m + \nu \cdot \mathbf{1}) 
\underline{u}^r \leq g(\underline{u})
\end{equation*}
when choosing $r$ such that $(1 - \delta m + \nu \cdot \mathbf{1}) \leq g(%
\underline{u})/\underline{u}^r$ on $(0,1)\times \Gamma$, which is possible
by assumption on $g$ and \eqref{1002}. Hence $\underline{u}$ is a
subsolution of \eqref{ell-dyn-classic} if $r$ is appropriate.

Now take $u_0\in \mathcal{X}_\delta$ with $u_0 \geq \underline{u}|_{t=0}$ on 
$\overline{\Omega}$. Let $u$ be the corresponding classical solution of %
\eqref{ell-dyn-classic}. Then $u \geq \underline{u}$ on $\overline{\Omega}$
by Lemma \ref{comparison}, as long as $u$ exists. Thus $u$ blows up at $t=1$.
\end{proof}

In case $f \equiv 0$ we can refine the blow-up condition for $g$.

\begin{proposition}
\label{blow-up-prop-2} Let $d> 0$ and $\delta \geq 0$. Suppose that there is 
$\xi_0$ such that $g(\xi)>0$ for $\xi \geq \xi_0$, and that 
\begin{equation*}
\int_{\xi_0}^\infty \frac{d\xi}{g(\xi)} < \infty.
\end{equation*}
Then there is $C>0$ such that for all $u_0\in \mathcal{X}_\delta$ the
solution of 
\begin{equation}
\left\{ 
\begin{array}{ll}
\Delta u=0 & \emph{\text{in }}(0,T)\times \Omega , \\ 
\partial _{t}u_{\Gamma }-\delta \Delta _{\Gamma }u_{\Gamma }+d\partial _{\nu
}u= g(u_\Gamma) & \emph{\text{on }} (0,T)\times \Gamma , \\ 
u_\Gamma|_{t=0}=u_{0} & \emph{\text{on }}\Gamma ,%
\end{array}%
\right.  \label{blow-up-2}
\end{equation}
blows up in finite time.
\end{proposition}

\begin{remark}
Under the additional assumption that $g$ is entirely positive, the above
result was shown in \cite[Theorem 1]{Kir} for $\delta =0$.
\end{remark}

\begin{proof}
For a constant initial value $\underline{u}_0 > 0$ the solution of %
\eqref{blow-up-2} is given by the solution $\underline{u}$ of $u^{\prime}=
g(u)$ with $u|_{t=0} = \underline{u}_0$. If $\underline{u}_0$ is
sufficiently large, then it is well-known that the condition on $g$ implies
that $\underline{u}$ blows up in finite time. By Lemma \ref{comparison}, any
solution of \eqref{blow-up-2} with initial value $u_0 \geq \underline{u}_0$
blows up as well.
\end{proof}

\subsection{Global existence}

\label{globals} We now return to the slightly more general assumptions %
\eqref{classical-assum} with variable diffusion coefficients. First we
refine the blow-up conditions and show that for both types of boundary
conditions an $L^\infty$-bound for $u_\Gamma$ suffices for global existence.

\begin{lemma}
\label{ge2} Let \eqref{classical-assum} hold, and assume that for $u_0 \in 
\mathcal{X}_\delta$ the solution $u$ of \eqref{ell-dyn-classic} satisfies 
\begin{equation*}
u_\Gamma \in L^\infty((0,t^+)\times \Gamma).
\end{equation*}
Then $t^+ = \infty$.
\end{lemma}

\begin{proof}
Suppose $t^+ < \infty$. We show $u_\Gamma \in L^\infty(0,t^+; \mathcal{X}
_\delta)$ to derive a contradiction. In both cases $\delta \geq \delta_*$
and $\delta \equiv 0$, for $T< t^+$ we may use the variation of constants
formula to estimate as in the proof of \cite[Proposition 3.2.1]{CD}, 
\begin{align}
\sup_{t\in [0,T]} \|u_\Gamma(t)\|_{\mathcal{X}_\delta} &\, \leq C_{t^+}\big (%
1 + \sup_{t\in [0,T]} \|g(u_\Gamma(t))\|_{L^p(\Gamma)}  \notag \\
&\,\qquad \qquad \qquad \quad + \sup_{t\in [0,T]} \|d \partial_\nu \mathcal{R%
}_\lambda(f(\mathcal{D}_\lambda(u_\Gamma(t))),0)\|_{L^p(\Gamma)}\big).
\label{333}
\end{align}
By assumption, the second summand is bounded independent of $T < t^+$. For
the third summand we have by Lemma \ref{Diri-nonlinear-lemma} that 
\begin{align*}
\|d \partial_\nu \mathcal{R}_\lambda(f(\mathcal{D}_\lambda(u_\Gamma(t))),0)%
\|_{L^p(\Gamma)} &\, \leq C\|f(\mathcal{D}_\lambda(u_\Gamma(t)))\|_{L^p(%
\Omega)} \leq C_\eta \|u_\Gamma(t)\|_{W^{\eta,p}(\Gamma)},
\end{align*}
where $\eta > 0$ is small. Given $\varepsilon >0$, it follows from the
interpolation inequality and Young's inequality that 
\begin{equation*}
\|u_\Gamma(t)\|_{W^{\eta,p}(\Gamma)} \leq \varepsilon \|u_\Gamma(t)\|_{%
\mathcal{X}_\delta} + C_\eps \|u_\Gamma(t)\|_{L^p(\Gamma)} \leq \varepsilon
\sup_{t\in [0,T]} \|u_\Gamma(t)\|_{\mathcal{X}_\delta} + C_\eps.
\end{equation*}
For sufficiently small $\varepsilon$ we may absorb $\varepsilon \sup_{t\in
[0,T]} \|u_\Gamma(t)\|_{\mathcal{X}_\delta}$ into the left-hand side of %
\eqref{333}. We thus find a bound for $\sup_{t\in [0,T]} \|u_\Gamma(t)\|_{%
\mathcal{X}_\delta}$ that is independent of $T < t^+$. Hence $u_\Gamma \in
L^\infty(0,t^+; \mathcal{X} _\delta)$.
\end{proof}

\begin{remark}
It follows from the proof above that if $g$ grows asymptotically at most
polynomial, then $u_{\Gamma }\in L^{\infty }( 0,t^{+}; L^q(\Gamma))$ for
sufficiently large $q<\infty$ is already sufficient for global existence.
\end{remark}

Before continuing we need the following inequality of Poincar\'e-Young type.

\begin{lemma}
\label{poincare} For all $p\in (1,\infty)$ and $\varepsilon \in (0,1)$ there
is $\tau >0$ such that 
\begin{equation}
\Vert u\Vert _{L^{p}(\Gamma )}\leq \varepsilon \Vert \nabla u\Vert
_{L^{p}(\Omega )}+\varepsilon ^{-\tau }\Vert u\Vert _{L^{1}(\Gamma )},\qquad 
\text{for all }u\in W^{1,p}(\Omega ).  \label{interi}
\end{equation}
\end{lemma}

\begin{proof}
\emph{Step 1.} We use the Poincar\'e inequality proved in \cite[Lemma 3.1]%
{RT01} to estimate 
\begin{align*}
\|u\|_{L^p(\Omega)} \leq \|u - \frac{1}{|\Gamma|} \int_\Gamma u
\|_{L^p(\Omega)} + C\|u\|_{L^p(\Gamma)} \leq C\big (\|\nabla
u\|_{L^p(\Omega)} + \|u\|_{L^p(\Gamma)}\big).
\end{align*}
Thus $\|\nabla u\|_{L^p(\Omega)} + \|u\|_{L^p(\Gamma)}$ is an equivalent
norm on $W^{1,p}(\Omega)$.

\emph{Step 2.} By a scaling argument it suffices to prove the inequality for 
$\|u\|_{L^p(\Gamma)} = 1$. Suppose that there is no $\tau >0$ such that the
inequality holds for a given $\varepsilon\in (0,1)$. Then for any $k\in 
\mathbb{N}$ there is $u_k \in W^{1,p}(\Omega)$ such that 
\begin{equation*}
\|u_k\|_{L^p(\Gamma)} = 1 \geq \varepsilon \|\nabla u_k\|_{L^p(\Omega)} +
\varepsilon^{-k} \|u_k\|_{L^1(\Gamma)}.
\end{equation*}
It follows from this inequality and Step 1 that the resulting sequence $%
(u_k) $ is bounded in $W^{1,p}(\Omega)$. Since the trace operator is a
compact map from $W^{1,p}(\Omega)$ into $L^p(\Gamma)$ and into $L^1(\Gamma)$%
, we find a subsequence, again denoted by $(u_k)$, that converges in $%
L^p(\Gamma)$ and in $L^1(\Gamma)$ to some limit $u$. By assumption we have $%
\|u\|_{L^p(\Gamma)} = 1$. On the other hand, the inequality shows that $%
\|u_k\|_{L^1(\Gamma)} \leq \varepsilon^k$ for all $k$, such that $%
\|u\|_{L^1(\Gamma)} = 0$ and thus $u|_\Gamma = 0$. This is a contradiction.
\end{proof}

We verify an $L^\infty(\Gamma)$-bound for solutions of %
\eqref{ell-dyn-classic} under the assumption that 
\begin{equation}
g(\xi )\xi \leq c_{g}( \xi ^{2}+1 ) \quad \text{ for all }\xi \in \mathbb{R},
\label{sign-cond}
\end{equation}
where $c_{g}$ is a nonnegative constant. Observe that this sign condition
complements the sufficient condition from Proposition \ref{blow-up-prop} for
blow-up.

\begin{proposition}
\label{gl_st} Let \eqref{classical-assum} hold, and assume \eqref{sign-cond}%
. Then for all $u_0 \in \mathcal{X}_\delta$ the classical solution of %
\eqref{ell-dyn-classic} exists globally in time, i.e., $t^{+}=\infty$.
\end{proposition}

\begin{proof}
We suppose that $t^+ <\infty$ and show $u_\Gamma \in L^\infty((0,t^+)\times
\Gamma)$ to derive a contradiction to Lemma \ref{ge2}.

\emph{Step 1.} Let $T < t^+$. By an iteration argument we will first show
that 
\begin{equation}
\Vert u_\Gamma \Vert _{L^{\infty }((0,T)\times \Gamma )}\leq C\,\max \left(
\Vert u_{0}\Vert _{L^{\infty }(\Gamma )},\Vert u_\Gamma \Vert _{L^{\infty
}(0,T;L^{2}(\Gamma ))}\right),  \label{moser}
\end{equation}%
where $C$ is independent of $u_\Gamma$ and $T$. Let $k\in \mathbb{N}$, fix $%
t\in (0,T)$ and write $u=u_\Gamma = u(t,\cdot )$. We multiply the equation
on $\Gamma $ by $u^{2^{k}-1}$ and integrate by parts on $\Gamma $ to obtain 
\begin{align*}
\frac{d}{dt}\int_{\Gamma }u^{2^{k}}dS& \,=-(2^{k}-1)2^{2-k} \int_{\Gamma
}\delta |\nabla _{\Gamma }\left( u^{2^{k-1}}\right) |^{2}dS \\
& \qquad + 2^k \int_{\Gamma }g(u)u^{2^{k}-1}dS-2^{k}\int_{\Gamma }d \partial
_{\nu }u u^{2^{k}-1}dS.
\end{align*}%
Multiplying the equation on $\Omega $ by $u^{2^{k}-1}$ gives 
\begin{align*}
-2^{k}d\int_{\Gamma }\partial _{\nu }u u^{2^{k}-1}dS & \, =
-(2^{k}-1)2^{2-k}\int_{\Omega }d |\nabla \left( u^{2^{k-1}}\right) |^{2}dx \\
&\, \qquad + 2^{k}\int_{\Omega }\left( f(u)u^{2^{k}-1}-\lambda
u^{2^{k}}\right) dx.
\end{align*}%
Using $-(2^{k}-1)2^{2-k}\leq -2$, that $f$ is globally Lipschitzian and that 
$\lambda\geq c_f$, we obtain 
\begin{align}
\frac{d}{dt}\int_{\Gamma }u^{2^{k}}dS& \,\leq -2d_* \int_{\Omega } |\nabla
\left( u^{2^{k-1}}\right) |^{2}dx  \notag \\
& \,\qquad \quad +2^{k}\int_{\Omega }\left( f(u)u^{2^{k}-1}-\lambda
u^{2^{k}}\right) dx+2^{k}\int_{\Gamma }g(u)u^{2^{k}-1}dS  \notag \\
&\, \leq - 2 d_*\int_{\Omega }|\nabla \left( u^{2^{k-1}}\right) |^{2}dx +C
\,2^{k} \int_{\Gamma}u^{2^{k}}dS + C\, 2^k.  \label{5bis}
\end{align}
Given $\varepsilon >0$, it follows from Lemma \ref{poincare} that there is $%
\tau >1$ such that 
\begin{equation*}
- \int_\Omega |\nabla v|^2\, dx \leq - \varepsilon^{-1} \int_\Gamma v^2 \,
dS + \varepsilon^{-\tau} \left ( \int_\Gamma |v| \, d S\right)^2.
\end{equation*}
Choosing $\varepsilon = \delta 2^{-k}$ with sufficiently small $\delta >0$,
we obtain that 
\begin{equation}  \label{moser5}
\frac{d}{dt}\int_{\Gamma }u^{2^{k}}\, dS \leq - 2^k \int_{\Gamma }u^{2^{k}}
\, d S + C 2^{k\tau}\int_{\Gamma }u^{2^{k-1}} \, d S + C 2^{k}, \qquad k\in 
\mathbb{N}.
\end{equation}
Now \eqref{moser} follows from a standard Moser-Alikakos iteration procedure
as presented e.g. in \cite[Proposition 9.3.1]{CD} (see also \cite[Lemma 5.5.3%
]{Mey10}).

\emph{Step 2.} Set $\varphi=\|u_\Gamma\|_{L^2(\Gamma)}^2$. Employing %
\eqref{5bis} with $k=1$, we get $\varphi^{\prime}\leq C_{1}\varphi +C_{2}$,
which we can integrate to 
\begin{equation*}
\varphi \left( t\right) \leq C_{1}\int_{0}^{t}\varphi (s)ds+\left(
tC_{2}+\varphi (0)\right) ,\qquad t\in (0,T).
\end{equation*}%
Thus, by Gronwall's inequality, 
\begin{equation*}
\|u_\Gamma\|_{L^2(\Gamma)}^2 \leq \left( tC_{2}+\int_{\Gamma
}u_{0}^{2}dS\right) e^{C_{1}t},\qquad t\in (0,T).
\end{equation*}%
Hence $u_\Gamma \in L^\infty(0,t^+; L^2(\Gamma))$, and therefore $u_\Gamma
\in L^\infty((0,t^+) \times \Gamma )$ by \eqref{moser}.
\end{proof}

Combining the Propositions \ref{blow-up-prop} and \ref{gl_st} gives the
following.

\begin{theorem}
Let \eqref{classical-assum} hold, assume $d \equiv d_*$ and $\delta \equiv
\delta_*\geq 0$, and that 
\begin{equation*}
g(\xi) \sim \rho |\xi|^{q-1} \xi \qquad \text{as }|\xi|\to \infty,
\end{equation*}
for some $\rho\in \mathbb{R}$ and $q>0$. Then for all $u_0 \in \mathcal{X}%
_\delta$ the problem \eqref{ell-dyn-classic} has a unique global classical
solution if and only if either $\rho \leq 0$ or $q\leq 1$.
\end{theorem}

As for blow-up, we refine the sufficient conditions on $g$ for global
existence in case of the Laplace equation. We argue as in \cite[Theorem 6.1]%
{CE02}, where the case $\delta \equiv 0$ was considered. The condition below
complements the one of Proposition \ref{blow-up-prop-2}.

\begin{proposition}
Let \eqref{classical-assum} hold, assume $d \equiv d_*$ and $\delta \equiv
\delta_*\geq 0$, and that $|g|\leq \gamma$, where $\gamma \in C(\mathbb{R},
(0,\infty))$ is such that 
\begin{equation*}
\int_0^\infty \frac{ds}{\gamma(s)} = \int_{-\infty}^0 \frac{ds}{\gamma(s)} =
\infty.
\end{equation*}
Then for all $u_0 \in \mathcal{X}_\delta$ the problem \eqref{blow-up-2} has
a unique global classical solution.
\end{proposition}

\begin{proof}
Suppose that $t^+ < \infty$, and let $\xi(t)$, $\zeta(t)\in \overline{\Omega}
$ be such that 
\begin{equation*}
m(t):= \min_{x\in \overline{\Omega}} u(t,x) = u(t,\xi(t)), \qquad M(t):=
\max_{x\in \overline{\Omega}} u(t,x) = u(t,\zeta(t)).
\end{equation*}
It follows from \cite[Theorem 3.5]{GT} that for each $t\in (0,t^+)$ we have $%
\xi(t), \zeta(t) \in \Gamma$. Thus $\partial_\nu u(t,\xi(t)) < 0$ and $%
\partial_\nu u(t,\zeta(t)) > 0$ by \cite[Lemma 3.4]{GT}. By \cite[Theorem 2.2%
]{CE02}, the function $m$ is almost everywhere differentiable on $(0,t^+)$
with $\partial_t m = (\partial_t u)(t,\xi(t)).$ Using that $\Delta_\Gamma
u(t,\xi(t)) \geq 0$, which can be seen as in the proof of Lemma \ref%
{comparison}, we get 
\begin{align*}
\partial_t m(t) = \delta \Delta_\Gamma u(t,\xi(t)) - \partial_\nu
u(t,\xi(t)) + g( u(t,\xi(t))) \geq - \gamma ( m(t))
\end{align*}
for a.e. $t\in (0,t^+)$. In the same way we obtain $\partial_t M(t) \leq
\gamma ( M(t))$ for a.e. $t\in (0,t^+).$ Now the very same arguments as in
the proof of \cite[Theorem 3.1]{CE02} provide a contradiction to the
assumption $t^+ < \infty$.
\end{proof}

%

\subsection{Global attractors}

\label{attractors} Suppose that \eqref{classical-assum} and \eqref{sign-cond}
hold true. Then by the above results, \eqref{ell-dyn-classic} generates a
compact global solution semiflow 
\begin{equation*}
S_{\delta }(t;u_{0}):=u(t;u_{0})
\end{equation*}%
of smooth solutions in the phase space $\mathcal{X}_{\delta }$. Let $%
F^{\prime }=f$ and $G^{\prime }=g$. Then we may differentiate 
\begin{equation*}
\mathcal{E}(u):=\frac{1}{2}\int_{\Omega }d |\nabla u|^{2}\,dx+\frac{1}{2}%
\int_{\Gamma }\delta |\nabla _{\Gamma }u_{\Gamma }|^{2}\,dS-\int_{\Omega
}(F(u)-\frac{\lambda }{2}u^{2})\,dx-\int_{\Gamma }G(u_{\Gamma })\,dS
\end{equation*}%
with respect to time, to obtain 
\begin{equation}
\partial _{t}\mathcal{E}(u)=-\Vert \partial _{t}u_{\Gamma }\Vert
_{L^{2}(\Gamma )}^{2}.  \label{identity}
\end{equation}%
Thus $\mathcal{E}$ is a strict Lyapunov function for \eqref{ell-dyn-classic}%
, and the problem is of gradient structure. By \cite[Corollary 1.1.7]{CD},
for the existence of a global attractor it is left to show the boundedness
of the set of equilibria $E$ of \eqref{ell-dyn-classic}. To formulate a
sufficient condition for this, we note that by the global Lipschitz
continuity of $f$ there is a constant $\tilde{c}_{f}\in \mathbb{R}$ such
that 
\begin{equation}
f(\xi )\xi \leq \tilde{c}_{f}(\xi ^{2}+1),\qquad \xi \in \mathbb{R}.
\label{c_ftilde}
\end{equation}%
For the boundedness of the equilibria the parameters of the problem should
satisfy 
\begin{equation}
\frac{\Vert \sqrt{d}\nabla \psi \Vert _{L^{2}(\Omega )}^{2}+\Vert \sqrt{%
\delta }\nabla _{\Gamma }\psi \Vert _{L^{2}(\Gamma )}^{2}-(\tilde{c}%
_{f}-\lambda )\Vert \psi \Vert _{L^{2}(\Omega )}^{2}-c_{g}\Vert \psi \Vert
_{L^{2}(\Gamma )}^{2}}{\Vert \psi \Vert _{L^{2}(\Gamma )}^{2}}\geq \eta >0,
\label{rayleigh}
\end{equation}%
for all $\Theta _{\delta }:=\left\{ \psi \in W^{1,2}(\Omega )\,:\, \psi
|_{\Gamma }\in W^{1,2}(\Gamma )\text{ if }\delta \geq \delta _{\ast
}\right\} $.

\begin{lemma}
\label{equi-bounded} Assume \eqref{classical-assum}, \eqref{sign-cond}, %
\eqref{c_ftilde} and \eqref{rayleigh}. Then the set of equilibria $E\subset
C^\infty(\overline{\Omega})$ of \eqref{ell-dyn-classic} is bounded in $%
W^{2,p}(\Gamma)$ for $\delta\geq \delta_*$ and it is bounded in $%
W^{1,p}(\Gamma)$ for $\delta \equiv 0$.
\end{lemma}

\begin{proof}
\emph{Step 1.} Note that indeed $E\subset C^{\infty }(\overline{\Omega })$
by Proposition \ref{classic-diff}. Thus an equilibrium $u$ satisfies 
\begin{equation}
\lambda u-\text{div}(d\nabla u)=f(u)\quad \text{in }\Omega ,\qquad -\text{div%
}_{\Gamma }(\delta \nabla _{\Gamma }u)+d\partial _{\nu }u=g(u_{\Gamma
})\quad \text{on }\Gamma .  \label{equi}
\end{equation}%
Multiplying by $u$, integrating by parts and using \eqref{sign-cond}, %
\eqref{c_ftilde} and \eqref{rayleigh}, we get 
\begin{align*}
C& \,\geq \Vert \sqrt{d}\nabla u\Vert _{L^{2}(\Omega )}^{2}+\Vert \sqrt{%
\delta }\nabla _{\Gamma }u_{\Gamma }\Vert _{L^{2}(\Gamma )}^{2}-(\tilde{c}%
_{f}-\lambda )\Vert u\Vert _{L^{2}(\Omega )}^{2}-c_{g}\Vert u_{\Gamma }\Vert
_{L^{2}(\Gamma )}^{2} \\
& \,\geq \eta \Vert u_{\Gamma }\Vert _{L^{2}(\Gamma )}^{2},
\end{align*}%
with a constant $C$ independent of $u$. Hence $\sup_{u\in E}\Vert u_{\Gamma
}\Vert _{L^{2}(\Gamma )}<\infty $. Next we obtain from \eqref{moser5} that
there are $C,\tau >0$ such that 
\begin{equation*}
\int_{\Gamma }u_{\Gamma }^{2^{k}}\,dS\leq 2^{k\tau }\int_{\Gamma }u_{\Gamma
}^{2^{k-1}}\,dS+C
\end{equation*}%
for all $u\in E$ and $k\in \mathbb{N}$. Hence, by an iteration argument, 
\begin{equation*}
\Vert u_{\Gamma }\Vert _{L^{\infty }(\Gamma )}\leq C\big (1+\Vert
u_{\Gamma}\Vert _{L^{2}(\Gamma )}\big).
\end{equation*}
Therefore 
\begin{equation}
\sup_{u\in E}\Vert u_{\Gamma }\Vert _{L^{\infty }(\Gamma )}<\infty .
\label{Li}
\end{equation}

\emph{Step 2.} Suppose that $\delta \geq \delta _{\ast }$. Then for $u\in E$%
, \eqref{Li} gives 
\begin{align*}
\Vert u_{\Gamma }\Vert _{W^{2,p}(\Gamma )}& \,\leq C\big(\Vert u_{\Gamma
}\Vert _{L^{p}(\Gamma )}+\Vert \Delta u_{\Gamma }\Vert _{L^{p}(\Gamma )}\big)
\\
& \,\leq C\big(1+\Vert g(u_{\Gamma })\Vert _{L^{p}(\Gamma )}+\Vert \partial
_{\nu }u\Vert _{L^{p}(\Gamma )}\big)\leq C(1+\Vert u\Vert
_{H^{2-1/p,p}(\Omega )}\big).
\end{align*}%
Recall that $u=\mathcal{D}_{\lambda }(u_{\Gamma })$, where $\mathcal{D}%
_{\lambda }:W^{2-2/p,p}(\Gamma )\rightarrow H^{2-1/p,p}(\Omega )$ is
globally Lipschitzian by Lemma \ref{Diri-nonlinear-lemma}. Using the
interpolation inequality, Young's inequality and \eqref{Li}, for arbitrary $%
\varepsilon >0$ we get 
\begin{equation*}
\Vert u\Vert _{H^{2-1/p,p}(\Omega )}\leq C\big (1+\Vert u_{\Gamma }\Vert
_{W^{2-2/p,p}(\Gamma )}\big)\leq \varepsilon \Vert u_{\Gamma }\Vert
_{W^{2,p}(\Gamma )}+C_{\varepsilon},
\end{equation*}%
where $C_{\varepsilon }$ does not depend on $u\in E$. For small $\varepsilon 
$ we can thus absorb $\varepsilon \Vert u_{\Gamma }\Vert _{W^{2,p}(\Gamma )}$
into the left-hand side of the previous inequality to obtain $\sup_{u\in
E}\Vert u_{\Gamma }\Vert _{W^{2,p}(\Gamma )}<\infty $.

\emph{Step 3.} Now let $\delta \equiv 0$. Then for $u\in E$ we have by Lemma %
\ref{Diri-lemma} that 
\begin{align*}
\|u_\Gamma\|_{W^{1,p}(\Gamma)} &\, \leq C\big( \|u_\Gamma\|_{L^p(\Gamma)} +
\|\mathcal{N}_\lambda u_\Gamma \|_{L^p(\Gamma)}\big) \\
&\, \leq C\big(1 + \|\partial_\nu \mathcal{R}_\lambda
(f(u),0)\|_{L^p(\Gamma)}\big) \leq C(1 + \|u\|_{L^p(\Omega)}\big).
\end{align*}
Since $\|u\|_{L^p(\Omega)} \leq C\big (1+ \|u_\Gamma\|_{W^{1-1/p,p}(\Gamma)}%
\big )$ by Lemma \ref{Diri-nonlinear-lemma}, we may argue as above to obtain 
$\sup_{u\in E} \|u_\Gamma\|_{W^{1,p}(\Gamma)} < \infty$.
\end{proof}

Under the above assumptions it now follows from \cite[Corollary 1.1.7]{CD}
that the semiflow $S_{\delta }$ generated by \eqref{ell-dyn-classic} has a
global attractor $\mathcal{A}_{\delta }.$ To verify that $\mathcal{A}%
_{\delta }$ has finite Hausdorff dimension, we need the following.

\begin{lemma}
\label{diff-semiflow-reactive} Assume \eqref{classical-assum} and %
\eqref{sign-cond}. Then for each $t >0$ the time $t$ map $S_\delta(t;\cdot)$
belongs to $C^\infty(\mathcal{X }_\delta)$, and the derivative $D_2
S_\delta(t;\cdot)$ is compact on $\mathcal{X }_\delta$.
\end{lemma}

\begin{proof}
Recall that \eqref{ell-dyn-classic} may be rewritten into the form %
\eqref{abs-semi}. The superposition operator induced by $g$ belongs to $%
C^\infty(W^{s,p}(\Gamma), L^p(\Gamma))$ for all $s > \frac{n-1}{p}$. By
Lemma \ref{Diri-nonlinear-lemma}, the same is true for $u_\Gamma \mapsto
d\partial_\nu \mathcal{R}_\lambda(f(\mathcal{D}_\lambda(u_\Gamma)),0).$
Therefore $S_\delta(t; \cdot)$ is smooth on $\mathcal{X}_\delta $ by e.g. 
\cite[Corollary 3.4.5]{Henry}. Since $S_\delta(t;\cdot)$ is a compact map by
Theorem \ref{thm-ell-dyn} and Proposition \ref{compact-reactive}, also $D_2
S_\delta(t;\cdot)$ is compact.
\end{proof}

Since the global attractor $\mathcal{A}_{\delta }$ is by definition
invariant under $S_{\delta }(1;\cdot )$, it is a consequence of \cite[%
Chapter V, Theorem 3.2]{T} that $\mathcal{A}_{\delta }$ has finite Hausdorff
dimension.

We summarize the the results of this subsection as follows.

\begin{theorem}
\label{maingl} Assume \eqref{classical-assum}, \eqref{sign-cond}, %
\eqref{c_ftilde} and \eqref{rayleigh}. Then the solution semiflow $S_{\delta
}$ on $\mathcal{X}_{\delta }$ for \eqref{ell-dyn-classic} has a global
attractor $\mathcal{A}_{\delta }$, which is of finite Hausdorff dimension
and coincides with the unstable set of equilibria.
\end{theorem}

\begin{remark}
Another (more indirect) way to prove that $\mathcal{A}_{\delta }$ has finite
fractal dimension is to establish the existence of a more refined object
called exponential attractor $\mathcal{E}_{\delta }$, whose existence proof
is based on the so-called smoothing property for the differences of any two
solutions. This can be easily carried out in light of the assumptions for $%
f,g$, and the smoothness both in space and time for the solutions of (\ref%
{ell-dyn-classic}) (see Proposition \ref{classic-diff}, and Lemma \ref%
{uniform2} below). It is also worth mentioning that the above result also
holds for less regular functions in \eqref{classical-assum}.
\end{remark}

We conclude with a result that states a necessary and sufficient condition
such that \eqref{rayleigh} is satisfied. To this purpose, consider the
(self-adjoint) eigenvalue problem (see, e.g., \cite[Theorem 2]{VV09}) 
\begin{equation}
\left( \lambda -\widetilde{c}_{f}\right) \varphi -\text{div}\left( d\nabla
\varphi \right) =0\qquad \text{ in }\Omega ,  \label{wi}
\end{equation}%
with a boundary condition that depends on the eigenvalue $\xi $ explicitly,%
\begin{equation}
-\text{div}_{\Gamma }\left( \delta \nabla _{\Gamma }\varphi \right)
+d\partial _{\nu }\varphi -c_{g}\varphi =\xi \varphi \qquad \text{ on }%
\Gamma .  \label{wib1}
\end{equation}

\begin{proposition}
\label{1st}Let $\delta \geq 0$. Then inequality (\ref{rayleigh}) is
satisfied if and only if the first eigenvalue $\xi _{1}^{\delta }=\xi
_{1}^{\delta }\left( \Omega ,\widetilde{c}_{f},c_{g}\right) $ of (\ref{wi})-(%
\ref{wib1}) is positive.
\end{proposition}

\begin{proof}
We have that (see \cite[Section 6]{VV09}) 
\begin{equation*}
\xi _{1}^{\delta }=\inf_{\psi \in \Theta _{\delta },\psi \neq 0}\frac{\Vert 
\sqrt{d}\nabla \psi \Vert _{L^{2}(\Omega )}^{2}+\Vert \sqrt{\delta }\nabla
_{\Gamma }\psi \Vert _{L^{2}(\Gamma )}^{2}-(\tilde{c}_{f}-\lambda )\Vert
\psi \Vert _{L^{2}(\Omega )}^{2}-c_{g}\Vert \psi \Vert _{L^{2}(\Gamma )}^{2}%
}{\Vert \psi \Vert _{L^{2}(\Gamma )}^{2}}\;,
\end{equation*}
from which the assertion immediately follows.
\end{proof}

\begin{remark}
It is easy to see that if $\lambda >\widetilde{c}_{f}$ and $%
c_{g}<C_{P}=C_{P}\left( \Omega ,d,\lambda ,\widetilde{c}_{f}\right) $, where 
$C_{P}>0$ is the best constant in the following Poincar\'{e}-Sobolev type
inequality%
\begin{equation*}
C_{P}\Vert \psi \Vert _{L^{2}(\Gamma )}^{2}\leq \Vert \sqrt{d}\nabla \psi
\Vert _{L^{2}(\Omega )}^{2}+(\lambda -\widetilde{c}_{f})\Vert \psi \Vert
_{L^{2}(\Omega )}^{2},
\end{equation*}%
then we always have $\xi _{1}^{\delta }>0$.
\end{remark}

\subsection{Convergence to single equilibria}

\label{cte}

We shall finally be concerned with the asymptotic behavior of single
trajectories. We first give sufficient conditions where a single homogeneous
equilibrium is approached exponentially fast by every solution with respect
to the $L^{2}(\Gamma )$-norm.

\begin{proposition}
\label{hom_e}Assume \eqref{classical-assum} and that $f^{\prime }\leq \tilde{%
c}_{f}$, $g^{\prime }\leq c_{g}$ for $\tilde{c}_{f},c_{g}\in \mathbb{R}$,
such that \eqref{rayleigh} is valid. If \eqref{ell-dyn-classic} has a
homogeneous equilibrium $u_{\ast }\in \mathbb{R}$, then for all $u_{0}\in 
\mathcal{X}_{\delta }$ we have 
\begin{equation*}
\Vert u(t;u_{0})-u_{\ast }\Vert _{L^{2}(\Gamma )}\leq e^{-2\eta t}\Vert
u_{0}-u_{\ast }\Vert _{L^{2}(\Gamma )},\qquad t>0.
\end{equation*}
\end{proposition}

\begin{proof}
We first note that if \eqref{rayleigh} holds true, then at most one
homogeneous equilibrium can exist, since it is necessary that either $\tilde{%
c}_f - \lambda <0$ or $c_g < 0$.

It is straightforward to see that $g(\xi )\xi \leq (c_{g}+1)\xi ^{2}+C$,
such that every solution exists globally in time by Proposition \ref{gl_st}.
Let $w=u(\cdot ,u_{0})-u_{\ast }$. Testing the equations for $w$ with $w$
itself, we get 
\begin{align*}
\frac{1}{2}\partial _{t}\Vert w_{\Gamma }\Vert _{L^{2}(\Gamma )}^{2}& \leq
-\Vert \sqrt{d}\nabla w\Vert _{L^{2}(\Omega )}^{2}-\Vert \sqrt{\delta }%
\nabla _{\Gamma }w_{\Gamma }\Vert _{L^{2}(\Gamma )}^{2} \\
&\qquad +(\tilde{c}_{f}-\lambda )\Vert w\Vert _{L^{2}(\Omega
)}^{2}+c_{g}\Vert w_{\Gamma }\Vert _{L^{2}(\Gamma )}^{2}.
\end{align*}%
Therefore $\partial _{t}\Vert w_{\Gamma }\Vert _{L^{2}(\Gamma )}^{2}\leq
-2\eta \Vert w_{\Gamma }\Vert _{L^{2}(\Gamma )}^{2}$ by \eqref{rayleigh},
and the result follows from Gronwall's inequality.
\end{proof}

We follow the approach of \cite{SW} to show the convergence of solutions to
single equilibria also in nontrivial situations, under the assumption that $%
f,g$ are real analytic. Thanks to Proposition \ref{classic-diff} we can work
with smooth solutions $u\in C^{\infty }\left( (0,t^{+})\times \overline{%
\Omega }\right) $.

In the situation of Theorem \ref{maingl}, the trajectory of any solution is
bounded in $\mathcal{X}_\delta$, and thus relatively compact. Combining this
with the gradient structure of \eqref{ell-dyn-classic}, which is due to %
\eqref{identity}, we obtain the following properties of the limit sets 
\begin{equation*}
\omega(u_0) = \{u_* \in \mathcal{X}_\delta \,:\, \exists\, t_k \nearrow
\infty \text{ such that }\; u(t_k;u_0) \to u_* \text{ as }k\to \infty\}
\end{equation*}
of trajectories (see, e.g., \cite[Chapter I and Chapter VII]{T}).

\begin{lemma}
\label{omega} Assume \eqref{classical-assum}, \eqref{sign-cond}, %
\eqref{c_ftilde} and \eqref{rayleigh}. Then for any $u_{0}\in \mathcal{X}%
_{\delta }$, the set $\omega (u_{0})$ is a nonempty, compact and connected
subset of $\mathcal{X}_{\delta }$. Furthermore, we have: \newline
\emph{(i)} $\omega (u_{0})$ is fully invariant for the corresponding
semiflow $S_{\delta }\left( t\right) $\ on $\mathcal{X}_{\delta }$; \newline
\emph{(ii)} $\mathcal{E}$ is constant on $\omega (u_{0})$; \newline
\emph{(iii)} $\emph{\text{dist}}_{\mathcal{X}_{\delta }}( S_{\delta
}(t)u_{0},\omega (u_{0})) \rightarrow 0$ as $t\rightarrow +\infty $; \newline
\emph{(iv)} $\omega (u_{0})$ consists of equilibria only.
\end{lemma}

The key to prove that each solution converges to a single equilibrium in
case when $f$ and $g$ are analytic is the following inequality of \L %
ojasiewicz-Simon type.

\begin{proposition}
\label{l3} Let $d\equiv d_*$ and $\delta \in \{0,1\}$. Assume that $f,g$ are
real analytic, $|f^{\prime }|\leq c_{f}$, $\lambda >c_{f}$ and that %
\eqref{sign-cond} and \eqref{rayleigh} are satisfied. Let $u_{\ast }\in
C^{\infty }(\overline{\Omega })$ be an equilibrium of \eqref{ell-dyn-classic}%
. Then there are constants $\theta \in (0,1/2)$ and $r>0$, depending on $%
u_{\ast }$, such that for any $u\in C^{2}\left( \overline{\Omega }\right) $
with $\Vert u-u_{\ast }\Vert _{H^{2}\left( \Omega \right) }+\delta \Vert
u_{\Gamma }-u_{\ast }\Vert _{H^{2}\left( \Gamma \right) }\leq r,$ we have 
\begin{align}
& \left\Vert \lambda u-d\Delta u-f(u)\right\Vert _{L^{2}(\Omega
)}+\left\Vert -\delta \Delta _{\Gamma }u_{\Gamma }+d\partial _{\nu
}u-g(u_{\Gamma })\right\Vert _{L^{2}\left( \Gamma \right) }  \label{LS} \\
& \qquad \geq |\mathcal{E}(u)-\mathcal{E}(u_{\ast })|^{1-\theta }.  \notag
\end{align}
\end{proposition}

\begin{proof}
Our proof follows closely that of \cite[Theorem 3.1]{SW} which only includes
the case $\delta =1$ (cf. also \cite{Wu} for $g\equiv 0$ and $\delta =0$).
We shall briefly mention the details below in the case when $\delta =0$ and $%
g$ is nontrivial. To this end, let us first set%
\begin{equation*}
V_{k}:=\left\{ \left( u,u_{\Gamma }\right) \in W^{k,2}\left( \Omega \right)
\times W^{k-1/2,2}\left( \Gamma \right) :u_{\Gamma }=u|_{ \Gamma }\right\} ,
\end{equation*}%
for $k\geq 1$ (by convention, we also let $V_{0}:=L^{2}\left( \Omega \right)
\times L^{2}\left( \Gamma \right) $). Here and below, for the sake of
simplicity of notation we will identify any function $u$ that belongs to $%
W^{k,2}\left( \Omega \right) $ with $\left( u,u_{\Gamma }\right) \in V_{k}$
such that $u_{\Gamma }\in W^{k-1/2,2}\left( \Gamma \right)$. Next, consider
the so-called Wentzell Laplacian, given by 
\begin{equation*}
A_{W}:=\left( 
\begin{array}{cc}
\lambda I-d\Delta & 0 \\ 
d\partial _{\nu } & 0%
\end{array}%
\right) ,
\end{equation*}%
with domain $D\left( A_{W}\right) =\left\{ u\in V_{1}:-\Delta u\in
L^{2}\left( \Omega \right) ,\text{ }\partial _{\nu }u\in L^{2}\left( \Gamma
\right) \right\} $, which we endow with its natural graph norm $\left\Vert
A_{W}\cdot \right\Vert _{V_{0}}$. In particular, $D\left( A_{W}\right)
=V_{2} $ provided that the boundary $\Gamma $ is sufficiently regular (see 
\cite{GW}, \cite{CGGR}). It is also well-known that $\left( A_{W},D\left(
A_{W}\right) \right) $ is self-adjoint and positive on $V_{0}$. By \cite%
{Gal0}, we also infer that there exists a complete orthonormal family $%
\left\{ \phi _{j}\right\} \subset V_{0}$, with $\phi _{j}\in D\left(
A_{W}\right) $, as well as a sequence of eigenvalues $0<\lambda _{1}\leq
\lambda _{2}\leq ... \leq \lambda _{j}\rightarrow \infty $, as $j\rightarrow
\infty $, such that $A_{W}\phi _{j}=\lambda _{j}\phi _{j},$ $j\in \mathbb{N}%
_{+}$. Moreover, by standard elliptic theory and bootstrap arguments we have 
$\phi _{j}\in C^{\infty },$ for every $j\in \mathbb{N}_{+}$ provided that $%
\Gamma $ is sufficiently regular (see \cite[Appendix]{Gal0}). Let now $P_{m}$
be the orthogonal projector from $V_{0}$ onto $K_{m}:=span\left\{ \phi
_{1},...,\phi _{m}\right\} $. Following a similar strategy to \cite[%
(3.8)-(3.11)]{SW}, it is easy to show that%
\begin{equation}
\left( A_{W}u+\lambda _{m}P_{m}u,u\right) _{V_{1}^{\ast },V_{1}}\geq
C_{\lambda ,d}\left\Vert u\right\Vert _{V_{1}}^{2}+\frac{1}{4}\lambda
_{m}\left\Vert u\right\Vert _{V_{0}}^{2}  \label{1005}
\end{equation}%
holds for any $u\in V_{1}$ (for some positive constant $C_{\lambda ,d}>0$).

Let $\psi $ be a critical point of $\mathcal{E}(u)$. For any $\psi \in
C^{2}\left( \overline{\Omega }\right) $, we consider the following
linearized operator $v\in V_{2}\longmapsto L\left( v\right) $, analogous to 
\cite[(3.12)]{SW}, given by%
\begin{equation*}
L\left( v\right) :=\left( 
\begin{array}{cc}
\lambda I-d\Delta & 0 \\ 
d\partial _{\nu } & 0%
\end{array}%
\right) +\left( 
\begin{array}{cc}
-f^{^{\prime }}\left( v+\psi \right) & 0 \\ 
0 & -g^{^{\prime }}\left( v+\psi \right)%
\end{array}%
\right) ,
\end{equation*}%
with domain $\mathcal{D}:=D\left( A_{W}\right) =V_{2}$. We note that one can
associate with $L\left( 0\right) $ the following bilinear form $b\left(
u_{1},u_{2}\right) $ on $V_{1}\times V_{1}$, as follows:%
\begin{equation*}
b\left( u_{1},u_{2}\right) =\int_{\Omega }\left( \lambda u_{1}u_{2}+d\nabla
u_{1}\cdot \nabla u_{2}-f^{^{\prime }}\left( \psi \right) u_{1}u_{2}\right)
dx+\int_{\Gamma }\left( -g^{^{\prime }}\left( \psi \right) u_{1}u_{2}\right)
dS,
\end{equation*}%
for any $u_{1},u_{2}\in V_{1}$. As in \cite{CGGR}, it can be easily shown
that $\left( L\left( v\right) ,\mathcal{D}\right) $ is self-adjoint on $%
V_{0} $. Moreover, by (\ref{1005}) it is readily seen that the operator $%
L\left( 0\right) +\lambda _{m}P_{m}$ is coercive with respect to the
(equivalent) inner product of $H^{1}\left( \Omega \right) ,$ provided that%
\begin{equation}
\lambda _{m}>4\max \left\{ \left\Vert f^{^{\prime }}\left( \psi \right)
\right\Vert _{L^{\infty }\left( \Omega \right) },\left\Vert g^{^{\prime
}}\left( \psi \right) \right\Vert _{L^{\infty }\left( \Gamma \right)
}\right\} .  \label{1006}
\end{equation}%
Recalling that $\psi $ is sufficiently smooth, we note that condition (\ref%
{1006}) can always be achieved by choosing a sufficiently large $m$.

Next, consider the following operators:%
\begin{equation*}
\Pi :=\lambda _{m}P_{m}:V_{0}\rightarrow V_{0},\text{ }\mathcal{L}\left(
v\right) :\mathcal{D}\rightarrow V_{0},\text{ }\mathcal{L}\left( v\right)
h=\Pi h+L\left( v\right) h,
\end{equation*}%
for any $v\in \mathcal{D}$. Clearly, $\mathcal{L}\left( 0\right) :\mathcal{D}%
\rightarrow V_{0}$ is bijective on account of (\ref{1006}). The following
lemma is concerned with some regularity properties for the (unique) solution
of $\mathcal{L}\left( 0\right) h=w$, for some given $w=\left(
w_{1},w_{2}\right) \in V_{0}.$

\begin{lemma}
Let $w=\left( w_{1},w_{2}\right) \in W^{k-1,2}\left( \Omega \right) \times
W^{k-1/2,2}\left( \Gamma \right) $, for any $k\in \mathbb{N}$, $k\geq 1$.
Then the following estimate holds:%
\begin{equation}
\left\Vert h\right\Vert _{V_{k+1}}\leq C\left( \left\Vert w_{1}\right\Vert
_{W^{k-1,2}\left( \Omega \right) }+\left\Vert w_{2}\right\Vert
_{W^{k-1/2,2}\left( \Gamma \right) }\right) .  \label{1007}
\end{equation}%
Moreover, it holds%
\begin{equation}
\left\Vert h\right\Vert _{C^{0,\gamma }\left( \overline{\Omega }\right)
}\leq C\left( \left\Vert w_{1}\right\Vert _{L^{p}\left( \Omega \right)
}+\left\Vert w_{2}\right\Vert _{L^{q}\left( \Gamma \right) }\right) ,
\label{1008}
\end{equation}%
as long as $w_{1}\in L^{p}\left( \Omega \right) $ with $p>n$ and $w_{2}\in
L^{q}\left( \Gamma \right) $ with $q>n-1.$ The constant $C>0$ is independent
of $k.$
\end{lemma}

Indeed, writing the equation $\mathcal{L}\left( 0\right) h=w$ in the form%
\begin{equation}
\left\{ 
\begin{array}{l}
\lambda h-d\Delta h=q_{1}:=f^{^{\prime }}\left( \psi \right) h-\Pi
h+w_{1},\quad \text{ a.e. in }\Omega , \\ 
d\partial _{\nu }h=q_{2}:=g^{^{\prime }}\left( \psi \right) h-\left( \Pi
h\right) _{\mid \Gamma }+w_{2},\quad \text{ a.e. on }\Gamma ,%
\end{array}%
\right.  \label{ellh}
\end{equation}%
we have $\left\Vert h\right\Vert _{V_{1}\cap \mathcal{D}}\leq C\left\Vert
w\right\Vert _{V_{0}}$ since $\mathcal{L}\left( 0\right) :\mathcal{D}%
\rightarrow V_{0}$ is a bijection. Recalling the standard trace-regularity
theory and the fact that $\psi \in C^{\infty }$, and applying the $W^{l,2}$
regularity theorem (see, e.g., \cite[II, Theorem IV.5.1]{LM}) to (\ref{ellh}%
) with $l=2,3,...,$ we have%
\begin{equation*}
\left\Vert h\right\Vert _{V_{l}}\leq C\left\Vert h\right\Vert
_{W^{l,2}\left( \Omega \right) }\leq C\left( \left\Vert q_{1}\right\Vert
_{W^{l-2,2}\left( \Omega \right) }+\left\Vert h\right\Vert _{W^{l-1,2}\left(
\Omega \right) }+\left\Vert q_{2}\right\Vert _{W^{l-3/2,2}\left( \Gamma
\right) }\right).
\end{equation*}%
From this (\ref{1007}) immediately follows by exploiting a standard
iteration argument for $l\geq 2$. The proof of (\ref{1008}) is contained in 
\cite{Wa} (see, also, \cite{Wa2} for the proof of the same bound in $%
L^{\infty }\left( \Omega \right) \times L^{\infty }\left( \Gamma \right) $%
-norm).

Exploiting now the results of the preceding lemma, the proof of the
(extended) \L ojasiewicz-Simon inequality (\ref{LS}) can be reproduced from
that of \cite[Theorem 3.1 and Lemma 3.2]{SW} with no essential modifications.
\end{proof}

To apply the proposition we need that the solutions converge to the limit
set in a norm that is stronger than that of $\mathcal{X}_\delta$. This will
be a consequence of the next lemma. Let $B_{\mathcal{X}_{\delta }}(R)$ be
the ball in $\mathcal{X}_{\delta }$ centered at the origin with radius $R>0$.

\begin{lemma}
\label{uniform2} Assume \eqref{classical-assum}, \eqref{sign-cond}, %
\eqref{c_ftilde} and \eqref{rayleigh}. Then for $k\in \mathbb{N}_0$ and $R >
0$ there is a constant $C= C(k,R) > 0$ such that 
\begin{equation*}
\sup_{u_0 \in B_{\mathcal{X}_\delta}(R) } \sup_{t \geq 1}
\|S_\delta(t;u_0)\|_{C^k(\overline{\Omega})} \leq C.
\end{equation*}
\end{lemma}

\begin{proof}
By Theorem \ref{maingl}, the semiflow has global attractor in $\mathcal{X}%
_\delta$. Therefore, using Lemma \ref{Diri-nonlinear-lemma}, 
\begin{equation*}
\|u\|_{L^\infty(\mathbb{R}_+; H^{2-1/p,p}(\Omega))} \leq C(R) \qquad \text{%
if } \delta \geq \delta_*,
\end{equation*}
\begin{equation*}
\|u\|_{L^\infty(\mathbb{R}_+; W^{1,p}(\Omega))} \leq C(R)\qquad \text{if }
\delta \equiv 0,
\end{equation*}
for all $u= u(\cdot; u_0)$ with $u_0 \in B_{\mathcal{X}_\delta}(R)$. Let us
consider the case $\delta \equiv 0$. We use the representation 
\begin{equation}  \label{msol}
u_\Gamma(t) = e^{-t\mathcal{N}_\lambda}u_0 + e^{- \cdot \mathcal{N}_\lambda
} *\big (g(u_\Gamma) - d \partial_\nu \mathcal{R}_\lambda (f(u),0)\big )(t),
\qquad t > 0,
\end{equation}
and assume, inductively, that 
\begin{equation*}
\Vert u\Vert _{L^{\infty }(1,\infty ;W^{k,p}(\Omega ))}\leq C(k,R)
\end{equation*}%
for some $k\in \mathbb{N}$. Then $\Vert g(u_{\Gamma })\Vert _{L^{\infty
}(1,\infty ;W^{k-1/p,p}(\Gamma ))}\leq C(k,R)$, and further, using \cite[%
Theorem 13.1]{Amann84}, 
\begin{equation*}
\Vert d\partial _{\nu }\mathcal{R}_{\lambda }(f(u),0)\Vert _{L^{\infty
}(1,\infty ;W^{k-1/p,p}(\Gamma ))}\leq C(k)\Vert f(u)\Vert _{L^{\infty
}(1,\infty ;W^{k-1,p}(\Omega ))}\leq C(k,R).
\end{equation*}%
Considering \eqref{msol} as an identity on $W^{k-1/p,p}(\Gamma )$, it
follows from \cite[Proposition 4.4.1]{Lun95} that $\Vert u_{\Gamma }\Vert
_{L^{\infty }(1,\infty ;W^{k+1-1/p,p}(\Gamma ))}\leq C(k+1,R).$ Moreover, by 
\cite[Theorem 13.1]{Amann84}, we have 
\begin{align*}
\Vert u\Vert _{L^{\infty }(1,\infty ;W^{k+1,p}(\Omega ))}& \,\leq C\big(%
\Vert f(u)\Vert _{L^{\infty }(1,\infty ;W^{k-1,p}(\Omega ))}+\Vert u_{\Gamma
}\Vert _{L^{\infty }(1,\infty ;W^{k+1-1/p,p}(\Gamma ))}\big) \\
& \,\leq C(k+1,R).
\end{align*}%
The asserted estimate in case $\delta \equiv 0$ now follows from Sobolev's
embeddings. The arguments in the case $\delta \geq \delta _{\ast }$ are
similar.
\end{proof}

We can now prove the main result of this subsection.

\begin{theorem}
\label{main2} Let $p\in (n,\infty )$, $d >0$ and $\delta \in \{0,1\}$.
Assume that $f,g$ are real analytic, $|f^{\prime }|\leq c_{f}$, $\lambda
>c_{f}$ and that $g$ satisfies \eqref{sign-cond} and \eqref{rayleigh}. Then
for any given initial datum $u_{0}\in \mathcal{X}_{\delta }$ the
corresponding solution $u\left( t;u_{0}\right) =S_{\delta }(t;u_{0})$ of (%
\ref{ell-dyn-classic}) exists globally in time and converges to a single
equilibrium $u_{\ast }$ in the topology of $\mathcal{X}_{\delta } $. More
precisely,%
\begin{equation}
{\lim_{t\rightarrow +\infty }}\left( \Vert u(t;u_{0})-u_{\ast }\Vert _{%
\mathcal{X}_{\delta }}+\Vert \partial _{t}u(t;u_{0})\Vert _{L^{2}\left(
\Gamma \right) }\right) =0.  \label{5.7}
\end{equation}
\end{theorem}

\begin{proof}
\emph{Step 1.} From Theorem \ref{maingl} and Lemma \ref{omega} we know that
the solution $u= u(\cdot;u_0)$ is smooth in space and time, exists globally
and that the corresponding trajectory converges to the set of equilibria in $%
\mathcal{X}_\delta$. By Lemma \ref{uniform2}, the trajectory is also bounded
in, say, $W^{3,p}(\Gamma)$. Since $\omega (u_{0})\subset C^{\infty }(%
\overline{\Omega }),$ we can apply the interpolation inequality (\ref%
{interpb}) with suitable $\theta\in (0,1)$ to obtain that 
\begin{equation}
\text{dist}_{\mathcal{V}_{\delta }}(u(t;u_{0}),\omega (u_{0}))\rightarrow
0\qquad \text{as }t\rightarrow +\infty ,  \label{prop3_bis}
\end{equation}%
where we have set $\mathcal{V}_{0}:=H^{2}\left( \Omega \right) $ if $\delta
=0$, and $\mathcal{V}_{1}:=H^{2}\left( \Omega \right) \oplus H^{2}\left(
\Gamma \right) $ if $\delta =1,$ respectively.

\emph{Step 2.} The function $t \mapsto \mathcal{E}(u(t))$ is decreasing and
bounded from below. Thus 
\begin{equation*}
\mathcal{E}_{\infty}:=\lim_{t\rightarrow \infty }\mathcal{E} ( u( t))
\end{equation*}
exists. If there is $t^\sharp$ with $\mathcal{E}( u( t^{\sharp })) = 
\mathcal{E}_\infty$, then $u$ is an equilibrium and there is nothing to
prove. Hence we may suppose that for all $t\geq t_{0}>0,$ we have $\mathcal{E%
}\left( u\left( t\right) \right) >\mathcal{E}_{\infty }.$ We first observe
that, by Lemma \ref{l3}, the functional $\mathcal{E}$ satisfies the {\L }%
ojasiewicz-Simon inequality (\ref{LS}) near every $u_{\ast }\in \omega
(u_{0}).$ Since $\omega (u_{0})$ is compact in $\mathcal{X}_{\delta }$, we
can cover it by the union of finitely many balls $B_{j}$ with centers $%
u_{\ast }^{j}$ and radii $r_{j},$ where each radius is such that (\ref{LS})
holds in $B_{j}$. It follows from Proposition \ref{l3} that there exist
uniform constants $\xi \in (0,1/2)$, $C_{L}>0$ and a neighborhood $U$ of $%
\omega (u_{0})$ in $\mathcal{X}_\delta$ such that (\ref{LS}) holds in $U$.
Thus, recalling (\ref{prop3_bis}), we can find a time $t_{0}\geq 1$ such
that $u ( t; u_0) $ belongs to $U$ for all $t\geq t_{0}.$ On account of (\ref%
{identity}) and (\ref{LS}) we obtain 
\begin{align}
& -\frac{d}{dt} ( \mathcal{E} ( u ( t)) -\mathcal{E}_{\infty }) ^{\xi } 
\notag \\
& =-\xi \partial _{t}\mathcal{E}( u( t)) ( \mathcal{E}( u( t)) -\mathcal{E}%
_{\infty }) ^{\xi -1}  \notag \\
& \geq \xi C_{L}\frac{\left\Vert \partial _{t}u_\Gamma\right\Vert
_{L^{2}\left( \Gamma \right) }^{2}}{\left\Vert \lambda u-d\Delta
u-f(u)\right\Vert _{L^{2}\left( \Omega \right) }+\left\Vert -\delta \Delta
_{\Gamma }u_{\Gamma }+d\partial _{\nu }u-g(u_\Gamma )\right\Vert
_{L^{2}\left( \Gamma \right) }}.  \notag
\end{align}%
Recalling (\ref{ell-dyn-classic}), we get 
\begin{equation}
-\frac{d}{dt}\left( \mathcal{E}\left( u\left( t\right) \right) -\mathcal{E}%
_{\infty }\right) ^{\xi }\geq C\left\Vert \partial _{t}u_\Gamma(t)
\right\Vert _{L^{2}\left( \Gamma \right) }.  \label{3.9bis}
\end{equation}%
Integrating over $( t_0,\infty) $ and using that $\mathcal{E}( u ( t ))
\rightarrow \mathcal{E}_{\infty }$ as $t \to \infty $, we infer that 
\begin{equation*}
\partial _{t}u_\Gamma \in L^{1}\left( t_0,\infty ;L^{2}\left( \Gamma \right)
\right).
\end{equation*}
\emph{Step 3.} Since $\partial_t u_\Gamma$ is uniformly continuous with
values in $L^2(\Gamma)$, it follows that $\|\partial_t
u_\Gamma(t)\|_{L^2(\Gamma)} \to 0$ as $t\to \infty$. Moreover, since by
Lemma \ref{omega} (iii) there are $t_k \nearrow \infty$ and $u_{\ast }\in
\omega \left( u_{0}\right) $ such that $u(t_{k})\rightarrow u_{\ast }$ in $%
\mathcal{X}_{\delta }$ as $k \to \infty $, the integrability of $\partial_t
u_\Gamma$ implies that $u_{\Gamma } ( t) \rightarrow u_*$ in $L^{2}( \Gamma)$
as $t \to \infty,$ and then in $\mathcal{X}_{\delta }$ as well. Hence $%
\omega ( u_{0}) =\{u_{\ast }\}$, and (\ref{5.7}) follows.
\end{proof}

\begin{remark}
One can also exploit (\ref{LS}) and (\ref{3.9bis}) to deduce a convergence
rate estimate in (\ref{5.7})\ of the form%
\begin{equation*}
\left\Vert u( t;u_{0}) -u_{\ast }\right\Vert _{L^{2}\left( \Gamma \right)
}\leq C(1+t)^{-1/(1-2\xi )},\qquad t >0,
\end{equation*}%
for some constants $C>0$, $\xi \in \left( 0,\frac{1}{2}\right) $ depending
on $u_{\ast }$. Taking advantage of the above (lower-order) convergence
estimate and the results of the previous subsections, one can also prove the
corresponding estimate in higher-order norms $W^{k,2}\left( \Omega \right) $%
, arguing, for instance, as in \cite{SW, Wu}.
\end{remark}

\end{document}